\documentclass[a4paper]{amsart}
\usepackage[T1]{fontenc}
\usepackage{a4wide}
\usepackage{mathtools}
\usepackage{algorithm}
\usepackage{algpseudocode}
\usepackage{tikz}
\usepackage{subcaption}
\usepackage{diagbox}
\usepackage{multirow}

\usepackage[parfill]{parskip}
\newcommand{\ts}{\textsuperscript}

\usepackage[sorting=none,url=false,maxbibnames=5,date=year]{biblatex}
\DeclareSourcemap{
    \maps{
        \map{ 
            \step[
                fieldsource=doi,
                match=\regexp{\{\\\_\}|\{\_\}|\\\_},
                replace=\regexp{\_}
            ]
        }
        \map{ 
            \step[
                fieldsource=url,
                match=\regexp{\{\\\_\}|\{\_\}|\\\_},
                replace=\regexp{\_}
            ]
        }
    }
}

\newtheorem{theorem}{Theorem}
\newtheorem{lemma}{Lemma}

\newcommand{\discr}[1]{\underline{#1}}

\makeatletter
\newcount\dirtree@lvl
\newcount\dirtree@plvl
\newcount\dirtree@clvl
\def\dirtree@growth{%
  \ifnum\tikznumberofcurrentchild=1\relax
  \global\advance\dirtree@plvl by 1
  \expandafter\xdef\csname dirtree@p@\the\dirtree@plvl\endcsname{\the\dirtree@lvl}
  \fi
  \global\advance\dirtree@lvl by 1\relax
  \dirtree@clvl=\dirtree@lvl
  \advance\dirtree@clvl by -\csname dirtree@p@\the\dirtree@plvl\endcsname
  \pgf@xa=0.7cm\relax
  \pgf@ya=-0.7cm\relax
  \pgf@ya=\dirtree@clvl\pgf@ya
  \pgftransformshift{\pgfqpoint{\the\pgf@xa}{\the\pgf@ya}}%
  \ifnum\tikznumberofcurrentchild=\tikznumberofchildren
  \global\advance\dirtree@plvl by -1
  \fi
}

\tikzset{
  dirtree/.style={
    growth function=\dirtree@growth,
    every node/.style={text centered,draw=black,thick,anchor=west,minimum height=1.5em,text height=0.8em},
    every child node/.style={anchor=west},
    edge from parent path={(\tikzparentnode\tikzparentanchor) |- (\tikzchildnode\tikzchildanchor)}
  }
}
\makeatother

\title[Diagonalization-Based PinT Preconditioners for Flow Control]{Diagonalization-Based Parallel-in-Time Preconditioners for Instationary Fluid Flow Control Problems}
\author{Bernhard Heinzelreiter and John W. Pearson}
\date{}

\begin{document}

\begin{abstract}
    We derive a new parallel-in-time approach for solving large-scale optimization problems constrained by time-dependent partial differential equations arising from fluid dynamics. The solver involves the use of a block circulant approximation of the original matrices, enabling parallelization-in-time via the use of fast Fourier transforms, and we devise bespoke matrix approximations which may be applied within this framework. These make use of permutations, saddle-point approximations, commutator arguments, as well as inner solvers such as the Uzawa method, Chebyshev semi-iteration, and multigrid. Theoretical results underpin our strategy of applying a block circulant strategy, and numerical experiments demonstrate the effectiveness and robustness of our approach on Stokes and Oseen problems. Noteably, satisfying results for the strong and weak scaling of our methods are provided within a fully parallel architecture.
\end{abstract}

\maketitle

\section{Introduction}
The field of partial differential equation (PDE)-constrained optimization problems has received great attention during the past few decades. This type of optimization problem arises widely in science and engineering, and has utility in many industrial processes. An overview of theory for such problems, their applications, and numerical methods for their solution can be found in e.g., \cite{Lions1971OptimalEquations,Troltzsch2009OptimalApplications,Hinze2009OptimizationConstraints}. Among a wide range of applications of PDE-constrained optimization, our focus is on fluid flow control problems, see e.g., \cite{Gunzburger2002PerspectivesOptimization,BieglerRealTime2007,Hazra2010Large-ScaleApplications}. Specifically, we will discuss the all-at-once solution of unsteady Stokes and Oseen control problems.

Since analytical solutions to these problems are not known in general, one is required to solve them numerically. However, their discretization often results in huge-scale systems of linear or linearized equations. The number of degrees of freedom sometimes scales poorly with the accuracy of the discretization. Off-the-shelf solvers, such as direct solvers for linear systems, often have storage requirements that are excessive for problems that are sufficiently finely-discretized to achieve high accuracy. In order to make a numerical solver feasible and effective, information about the linear system itself and the structure of the PDEs must frequently be taken into account when designing it. In recent years, preconditioned iterative methods have been successfully applied to PDE-constrained optimization problems \cite{Schoberl2007SymmetricProblems,Schoberl2011AProblems,Rees2010OptimalOptimization, Pearson2012AOptimization,Pearson2013FastProblems,Pearson2013FastProcesses,McDonald2018PreconditioningEquations}, and, in particular, bespoke robust preconditioners for stationary and instationary Stokes and Oseen problems have enabled the fast and robust solution of flow control problems \cite{Silvester2001EfficientFlow,Kay2002AEquations,Krendl2013StabilityProblems,Krendl2015EfficientEquations}. Recent research has made it possible to tackle huge-scale non-linear time-dependent problems and solve them to high accuracy, see e.g., \cite{Danieli2022Space-timeFlow,Leveque2022FastTime,Hinze2012AControl,Janssens2024Parallel-in-timeEquations}.

A key challenge in these fields is that there is a pressing need to model and simulate increasingly complicated and fine-scale physical phenomena, but in previous decades CPUs have not experienced a significant increase in speed and are also not expected to do so in the future~\cite{Sutter2005TheSoftware}. However, processing architectures with increasing numbers of concurrent processing units have become more available. This in turn has increased the interest in parallelization of numerical methods in order to increase the range of problems they are capable of coping with, to resolve many of these numerical challenges. These include {\it parallel-in-time} methods, which are numerical methods for solving time-dependent problems that are designed to be parallelizable; a good overview of such methods can be found in~\cite{Gander201550Integration}. A number of approaches within this family have been suggested. Methods such as Parareal~\cite{Lions2001Resolutionparareel} and ParaDiag~\cite{Gander2021ParaDiag:Technique} have been shown to accelerate the solution of evolution equations, but these must be adapted to become applicable when considering optimality systems, as they occur in time-dependent PDE-constrained optimization problems. Examples of parallel-in-time methods which have been successfully applied to optimization problems involving PDEs may be found in~\cite{Gander2020ParaOpt:Systems,Gotschel2019AnPDEs,Bouillon2023OnControl,Wu2020AEquations,Wu2020Diagonalization-basedProblems}.

In this work, we derive a new diagonalization-based parallel-in-time approach for fluid flow control problems, which allows their rapid solution by making use of the fast Fourier transform (FFT). With suitable permutations, we may then reformulate the problem to one of designing suitable inner solvers for systems of equations arising from individual time points, for which we derive bespoke, potent approximations. We make use of theory for saddle-point preconditioners, existing robust preconditioners for flow problems based on commutator arguments \cite{Silvester2001EfficientFlow,Kay2002AEquations}, and tailored routines for individual submatrices, resulting in fast and robust preconditioners for the all-at-once solution of the Stokes and Oseen problems. To our knowledge, this is the first parallelizable-in-time preconditioner for either of these problems, which demonstrates very good (strong and weak) scaling properties in a practical parallel architecture.


The paper is structured as follows. In Section~\ref{seq:ProblemFormulation}, the flow control problems we consider are introduced. Their first-order optimality conditions are stated and the resulting systems of discretized equations are derived. Moreover, basic tools for the numerical analysis of saddle-point systems are summarized for the general case, but also in the context of flow problems. Section~\ref{seq:PreconditionerStokes} discusses the diagonalization of the system of equations. It is presented how adaptations of existing preconditioners for flow control problems can be embedded into this framework. Section~\ref{seq:PreconditionerOseen} generalizes these approaches to the Oseen problem. In Section~\ref{seq:NumericalExperiments}, we present numerical results, including results run in parallel on problems with more than 76 million degrees of freedom. Finally, concluding remarks are given in Section~\ref{seq:Conclusion}.

\section{Problem formulation}
\label{seq:ProblemFormulation}

In this paper, we develop a numerical method for solving instationary Stokes and Oseen flow control problems. Specifically, we consider an optimization problem that aims to find a distributed control minimizing the least-squares distance of the velocity field from a desired velocity governed by a Stokes or an Oseen flow. Additionally, the problem is regularized by penalizing the control in the cost functional. Since the Stokes flow is a special case of the Oseen flow, we will first introduce the problem for the latter and then highlight the difference between the two settings.

The Oseen flow control problem can be formulated more concretely as the minimization of the functional
\begin{align}
    \min_{v, p, u} \quad J(v, u) = \frac{1}{2} \int_0^T \int_\Omega \| v(x, t) - v_d(x, t) \|^2 \mathrm{d}x \mathrm{d}t + \frac{\beta}{2} \int_0^T\int_\Omega \| u(x, t) \|^2 \mathrm{d}x \mathrm{d}t,
    \label{eq:OseenOptimization}
\end{align}
subject to the equations
\begin{align}
\left\{
\begin{alignedat}{6}
    \text{s.t.} &\quad &\frac{\partial v}{\partial t} - \nu \Delta v + w(x) \cdot \nabla v + \nabla p &= u + f(x, t) &\quad&\text{~~in }\Omega \times (0, T), \\
                &      & -\nabla \cdot v &= 0  &&\text{~~in } \Omega \times (0, T), \\
                &      & v(x, t) &= h(x, t) &&\text{~~on } \partial \Omega \times (0, T), \\
                &      & v(x, 0) &= v_0(x)  &&\text{~~in } \Omega,
\end{alignedat}
\right.
\label{eq:OseenFlow}
\end{align}
where $\Omega \subset \mathbb{R}^d$ is a bounded domain of dimension $d=2, 3$. The function $v$ represents the vector velocity field, $p$ the scalar pressure field, and $u$ the control variable. The desired velocity state is denoted as $v_d$. The function $w$ is referred to as the {\it wind}, which is a distinctive feature to the Oseen flow, and $\nu$ is the viscosity. We will assume that the wind is divergence-free, i.e., $\nabla \cdot w(x) = 0$ in $x \in \Omega$. The parameter $\beta > 0$ is the regularization parameter and determines the weight of the control costs. 
The Stokes flow control problem is achieved if no wind is present, i.e., $w = 0$.

Equation \eqref{eq:OseenFlow} is a set of partial differential equations in strong form. In order to construct finite element approximations, we will consider the weak formulation of these equations. This requires us to add the assumption that $\Omega$ is a Lipschitz domain and to consider the function spaces $v \in H^1(\Omega)^d$, $p \in L^2(\Omega)$, and $u \in L^2(\Omega)^d$. For a detailed discussion of the weak form for the Stokes problem, the reader is referred to \cite[Chapter~3]{Elman2014FiniteSolvers}. The Oseen flow can be treated similarly. For the reader's convenience, the following derivations and equations are presented in the strong form of the equations, however, these can be carried over to the weak form with only a few adaptations. 

\subsection{Optimality conditions}
The formulation~\eqref{eq:OseenOptimization}--\eqref{eq:OseenFlow} represents a convex optimization problem. This means it is sufficient to consider the first-order optimality conditions in order to solve the problem. These are given by the following set of equations:
\begin{align}
&\left\{
\begin{alignedat}{4}
    \mathrlap{\hspace{1em} \frac{\partial v}{\partial t} - \nu \Delta v + w(x) \cdot \nabla v + \nabla p} 
    \hphantom{-\frac{\partial \lambda}{\partial t} - \nu \Delta \lambda + w(x) \nabla \cdot \lambda + \nabla \mu}
    &= \frac{1}{\beta} \lambda + f(x, t) &&\quad\text{ in } \Omega \times (0, T), \\
    -\nabla \cdot v &= 0&&\quad\text{ in } \Omega \times (0, T), \\
    v(x, t) &= h(x, t)&&\quad\text{ on } \partial\Omega \times (0, T) \\
    v(x, 0) &= v_0(x)&&\quad\text{ on } \Omega,
\end{alignedat}\right. 
\label{eq:StateEquation}\\
&\left\{
\begin{alignedat}{4}
    -\frac{\partial \lambda}{\partial t} - \nu \Delta \lambda - w(x) \cdot \nabla \lambda + \nabla \mu
     &= v_d - v&&\quad\text{ in } \Omega \times (0, T), \\
    -\nabla \cdot \lambda &= 0&&\quad\text{ in } \Omega \times (0, T), \\
    \lambda(x, t) &= 0&&\quad\text{ on } \partial\Omega \times (0, T), \\
    \lambda(x, T) &= 0&&\quad\text{ on } \Omega.
\end{alignedat}\right.
\label{eq:AdjointEquation}
\end{align}
Equations \eqref{eq:StateEquation} and \eqref{eq:AdjointEquation} are the {\it state} and {\it adjoint equations}, respectively. When deriving the optimality conditions, one would additionally obtain the gradient equation
\begin{align*}
    \beta u - \lambda = 0,
\end{align*}
which has already been substituted in to obtain the above state equations. Overall, the solution of the optimization problem has been narrowed down to solving the coupled PDEs \eqref{eq:StateEquation} and \eqref{eq:AdjointEquation}. 

\subsection{Discretization}
The optimality conditions are discretized with a finite element method in space. The function approximation spaces have to be chosen carefully to achieve stability of the numerical method. We will restrict ourselves to two stable choices of elements, which are $P_2$-$P_1$-elements, so-called Taylor-Hood elements, and $Q_2$-$Q_1$-elements. The properties of these discretizations are discussed in \cite[p.~133 ff.\ and p.~149 ff.]{Elman2014FiniteSolvers}.
In time, the equations are discretized with the backward Euler method. For the temporal discretization, it has to be taken into account that the state equation is to be considered forward in time and the adjoint equation backward in time. The number of time points is $n_t + 1$ for some $n_t \in \mathbb{N}$, which leads to a timestep of $\tau = T / n_t$. This gives us the equations
\begin{align*}
    M \discr{v}^{(j + 1)} - M \discr{v}^{(j)} + \tau L \discr{v}^{(j + 1)} + \tau B^T \discr{p}^{(j + 1)} &= \frac{\tau}{\beta} M \discr{\lambda}^{(j + 1)} + \discr{h}^{(j + 1)} \quad\text{for } j \in \{0, \dots, n_t - 1\}, \\
    \tau B \discr{v}^{(j + 1)} &= 0 \quad\text{for } j \in \{0, \dots, n_t - 1\},\\
    M \discr{\lambda}^{(j - 1)} - M \discr{\lambda}^{(j)} + \tau L^T \discr{\lambda}^{(j - 1)} + \tau B^T \discr{\mu}^{(j - 1)} &= \discr{g}^{(j - 1)} + \tau M \left(\discr{v}^{(j - 1)}_d - \discr{v}^{(j - 1)} \right) \quad\text{for } j \in \{1, \dots, n_t\}, \\
    \tau B \discr{\mu}^{(j - 1)} &= 0 \quad\text{for } j \in \{1, \dots, n_t\},
\end{align*}
where superscript $j$ indicates the values at time $t = j \tau$. The variables $\discr{v}$, $\discr{p}$, $\discr{\lambda}$, and $\discr{\mu}$ denote the vectors containing the values of the spatially discretized functions. The matrix $M$ is the mass matrix in the velocity space, $L$ is the discretization of the weak form of $-\nu \Delta + w(x) \cdot \nabla$
and $B$ is the discretized negative divergence. The vectors $\discr{g}^{(\cdot)}$ and $\discr{h}^{(\cdot)}$ include boundary conditions and forcing terms. For $j = 0$ and $j = n_t$, we need to specify the initial and terminal conditions
\begin{align*}
    \discr{v}^{(0)} &= \discr{v}_0, \\
    \discr{\lambda}^{(n_t)} &= \underline{0}.
\end{align*}
It is important to note that in the above discretization the pressure field is unique only up to a constant, which means that the resulting linear system is singular. In order to prevent potential difficulties caused by this singularity, we transform the system to a regular one by fixing the pressure at one pre-defined point in space. This way the divergence matrix $B$ becomes full rank and the overall linear system invertible.

Assembling this into an all-at-once system leads to linear equations of the form
\begin{align*}
    \bar{\mathcal{A}} \begin{pmatrix}
        \discr{v} & \discr{p} & \discr{\lambda} & \discr{\mu}
    \end{pmatrix}^T = \begin{pmatrix}\discr{g} & \discr{h}\end{pmatrix}^T,
\end{align*}
where $\discr{v} = \begin{pmatrix} \discr{v}^{(1)} & \cdots & \discr{v}^{(n_t - 1)} \end{pmatrix}^T$ is the concatenation of the discretized velocity at all time points (apart from the first and last). The vectors $\discr{p}$, $\discr{\lambda}$, and $\discr{\mu}$ are defined similarly. The matrix is defined as 
\begin{align*}
\bar{\mathcal{A}} = \begin{pmatrix}
        \tau I_t \otimes M & 0 & E^T \otimes M + \tau I_t \otimes L^T & \tau I_t \otimes B^T \\
        0 & 0 & \tau I_t \otimes B & 0 \\
        E \otimes M + \tau I_t \otimes L & \tau I_t \otimes B^T & -\frac{\tau}{\beta} I_t \otimes M & 0 \\
        \tau I_t \otimes B & 0 & 0 & 0
    \end{pmatrix}.
\end{align*}
The matrix $I_t$ is the identity matrix of size $n_t - 1$, and $E \in \mathbb{R}^{(n_t - 1) \times (n_t - 1)}$ arises from the backward Euler discretization and has the form
\begin{align*}
    E = \begin{pmatrix}
        1 & & &   \\
        -1 & 1 & &  \\
         & \ddots & \ddots &    \\
         & & -1 & 1 
    \end{pmatrix}.
\end{align*}
The right-hand side $\begin{pmatrix}\discr{g} & \discr{h}\end{pmatrix}^T$ assembles all $\discr{g}^{(j)}$ and $\discr{h}^{(j)}$ and accounts for the forcing terms, boundary conditions, and the initial and terminal conditions. As we will only consider the discretized system in the following, the underlined notation for the discretized variables will be dropped for the sake of better readability. 

While $\bar{\mathcal{A}}$ results from the direct assembly of the discretized equations maintaining the original ordering, we will consider a different permutation. The linear system can be permuted to obtain a matrix of the form
\begin{align}
    \mathcal{A} = \begin{pmatrix}
        \tau I_t & E^T \\
        E & -\frac{\tau}{\beta} I_t
    \end{pmatrix} \otimes
    \begin{pmatrix}
        M & 0 \\
        0 & 0
    \end{pmatrix}
    +
    \tau \begin{pmatrix}
        0 & 0 \\
        I_t & 0
    \end{pmatrix} \otimes
    \begin{pmatrix}
        L & B^T \\
        B & 0
    \end{pmatrix}
    +
    \tau \begin{pmatrix}
        0 & I_t \\
        0 & 0
    \end{pmatrix} \otimes
    \begin{pmatrix}
        L^T & B^T \\
        B & 0
    \end{pmatrix},
    \label{eq:FinalA}
\end{align}
whose Kronecker product form makes upcoming manipulations clearer.

\subsection{Saddle-point systems}
Given that several of the linear systems under consideration exhibit a saddle-point structure, this section will briefly present key properties of generalized saddle-point systems. Saddle-point systems have been extensively studied over the past decades, with some of their findings being the basis for key ideas of this work. We are especially interested in their efficient numerical solution and will outline relevant preconditioners in this section. A detailed discussion of the analysis and numerical treatment of these systems can be found in \cite{Benzi2005NumericalProblems}. 

The (generalized) saddle-point systems discussed here are linear systems of the form
\begin{align*}
    \underbrace{\begin{pmatrix}
        \Phi & \Psi^T \\
        \Psi & -\Theta
    \end{pmatrix}}_{= \mathcal{H}}
    \begin{pmatrix}
        x_1 \\
        x_2
    \end{pmatrix} = 
    \begin{pmatrix}
        b_1 \\
        b_2
    \end{pmatrix}.
    \label{eq:SaddlePointSystem}
\end{align*}
It is assumed that the matrices $\Phi \in \mathbb{R}^{n \times n}$ and $\Theta^{m \times m}$ are symmetric positive definite and symmetric positive semi-definite, respectively. The Schur complement of $\mathcal{H}$ is defined by
\begin{align*}
    S = \Theta + \Psi \Phi^{-1} \Psi^T
\end{align*}
and forms the basis for a broad range of preconditioners for $\mathcal{H}$, since it enables the decomposition
\begin{align*}
    \mathcal{H} = \begin{pmatrix}
        I & 0 \\
        \Psi \Phi^{-1} & I
    \end{pmatrix}
    \begin{pmatrix}
        \Phi & 0 \\
        0 & S
    \end{pmatrix}
    \begin{pmatrix}
        I & \Phi^{-1} \Psi^T \\
        0 & -I
    \end{pmatrix} = 
    \begin{pmatrix}
        \Phi & 0 \\
        \Psi & S
    \end{pmatrix}
    \begin{pmatrix}
        I & \Phi^{-1} \Psi^T \\
        0 & -I
    \end{pmatrix},
\end{align*}
which in turn justifies the following choice of preconditioner:
\begin{align*}
\mathcal{P} = \begin{pmatrix}
    \Phi & 0 \\
    \Psi & S
\end{pmatrix}.
\end{align*}
The resulting preconditioned system has only two distinct eigenvalues $\lambda(\mathcal{P}^{-1}\mathcal{H}) = \{-1, 1\}$, which means that preconditioned Krylov subspace solvers, such as MINRES \cite{Paige1975SolutionEquations} or GMRES \cite{Saad1986GMRES:Systems}, converge within two iterations. In order to make the application of $\mathcal{P}^{-1}$ practical, an efficient application of $\Phi^{-1}$ and $S^{-1}$ is required, which is why instead of $\mathcal{P}$, we consider an approximation
\begin{equation}
\widehat{\mathcal{P}} = \begin{pmatrix}
    \widehat{\Phi} & 0 \\
    \Psi & \widehat{S}
\end{pmatrix}.
\label{eq:SaddlePointTriangularPreconditioner}
\end{equation}
While a good approximation of $\Phi$ can often be achieved with methods such as multigrid methods or iterative methods like Chebyshev semi-iteration \cite{Golub1961ChebyshevI,Wathen2008ChebyshevMatrix} when dealing with PDEs, the approximation of the Schur complement is a more delicate task. Since $S$ is dense in general and also expensive to compute, it is infeasible to assemble $S$ in most PDE-based applications. In certain cases, however, efficient approximations with good spectral properties can be achieved. Subsequently, we will introduce two of such approximations that will be leveraged later.

One Schur complement approximation that is used in this work was first presented in \cite{Pearson2012AOptimization,Pearson2013FastProblems} for systems that fulfill $n = m$ and $\Theta = \alpha \Phi$. The authors proposed to use the approximation
\begin{align}
    S \approx \widehat{S} = \left( \Psi + \sqrt{\alpha} \Phi \right) \Phi^{-1} \left( \Psi + \sqrt{\alpha} \Phi \right)^T.
\label{eq:SchurComplementPearsonWathen}
\end{align}
It was shown that the spectral properties are robust with respect to the parameter $\alpha$. More precisely, it was shown in \cite{Pearson2013FastProblems} that if $\Psi + \Psi^T$ is positive semi-definite, then
\begin{align*}
    \lambda\left( \widehat{S}^{-1} S \right) \in \left[\frac{1}{2}, 1\right].
\end{align*}
 
The other case in which we are particularly interested here is the saddle-point system
\begin{align*}
\mathcal{H} = \begin{pmatrix}
    L & B^T \\
    B & 0
\end{pmatrix}
\end{align*}
with the finite element matrices defined in the preceding section. A good approximation of the Schur complement is based on the idea of the commutator. This idea and its application to flow problems were proposed in e.g., \cite{Silvester2001EfficientFlow,Kay2002AEquations}, and an introductory overview is given in \cite[Section 9.2]{Elman2014FiniteSolvers}. The starting point of the argument is the assumption that the operator
\begin{align*}
    \mathcal{E} = \nabla \cdot \left(-\nu \Delta + w \cdot \nabla\right) - \left(-\nu \Delta + w \cdot \nabla\right)_p \nabla \cdot
\end{align*}
is small, which gives reason that
\begin{align*}
    S = B L^{-1} B^T \approx M_p L_p^{-1} K_p \approx \widehat{S}
\end{align*}
forms a good approximation.
The subscripts $p$ denote the corresponding differential operators and matrices in the pressure space and the matrix $K_p$ represents the finite element stiffness matrix arising from the negative Laplacian. In the case of the Stokes problem, it can be shown that the eigenvalues of the preconditioned Schur complement are bounded from above by $1$. Theoretical results on the lower bound relate to (the square of) an inf--sup constant, which is generally not known analytically. Experiments, however, have verified that in practice $\widehat{S}$ indeed approximates well (in a spectral sense) the Schur complement, cf.~\cite[p.~175]{Elman2014FiniteSolvers}. Hence, real numbers $a, b > 0$ can be found such that
\begin{align*}
    \lambda\left( \widehat{S}^{-1} S \right) \in [a, b],
\end{align*}
where $a$ and $b$ are robust with respect to the discretization and the model parameters. Thus, the block triangular preconditioner \eqref{eq:SaddlePointTriangularPreconditioner} with an exact (1,1)-block and the above Schur complement approximation has the spectral property
\begin{align*}
\left|\lambda \left(
\widehat{\mathcal{P}}^{-1} \mathcal{H}
\right)\right| \in [a, b].
\end{align*}
For the Oseen problem, such bounds cannot be established since the differential operator has complex eigenvalues in general. The preconditioner, nonetheless, provides good clustering of the eigenvalues of the preconditioned system. 

\section{Preconditioners for the Stokes control problem}
\label{seq:PreconditionerStokes}
We aim to solve equation~\eqref{eq:FinalA} with a Krylov subspace iterative solver. To apply such methods effectively, one is required to provide efficient preconditioners. We are interested in how a preconditioner can be designed to allow for parallelization among the timesteps to make systems with larger time horizons tractable. This would provide a so-called {\it parallel-in-time} preconditioner. Over the last decades, various parallel-in-time approaches have been developed for optimality systems, see e.g., \cite{Gander2020ParaOpt:Systems,Gotschel2019AnPDEs,Bouillon2023OnControl}. We will now focus on diagonalization-based approaches for the systems under consideration.

\subsection{Block diagonalization with FFT}
\label{sec:DiagonalizationFFT}
First, our goal is to obtain a diagonalized representation of $\mathcal{A}$ that can be applied cheaply. There is no inherent potential for diagonalizing (by blocks) the matrix~$\mathcal{A}$ as it is, which is why we consider a slightly perturbed matrix $\mathcal{P}_C$. The preconditioner is achieved by replacing $E$ with a low-rank perturbation $C$
\begin{align*}
    \mathcal{P}_C = \begin{pmatrix}
        \tau I_t & C^T \\
        C & -\frac{\tau}{\beta} I_t
    \end{pmatrix} \otimes
    \begin{pmatrix}
        M & 0 \\
        0 & 0
    \end{pmatrix}
    +
    \tau \begin{pmatrix}
        0 & 0 \\
        I_t & 0
    \end{pmatrix} \otimes
    \begin{pmatrix}
        L & B^T \\
        B & 0
    \end{pmatrix}
    +
    \tau \begin{pmatrix}
        0 & I_t \\
        0 & 0
    \end{pmatrix} \otimes
    \begin{pmatrix}
        L^T & B^T \\
        B & 0
    \end{pmatrix},
\end{align*}
where $C$ is given by
\begin{align*}
    C = \begin{pmatrix}
        1 & & & -1  \\
        -1 & 1 & &  \\
         & \ddots & \ddots &  \\
         & & -1 & 1 
    \end{pmatrix} \approx E.
\end{align*}
An interpretation of this is that the original Stokes problem is replaced {\it for the purposes of deriving preconditioners} with its periodic-in-time analogue. In \cite{Krendl2015EfficientEquations}, this problem is analyzed and its diagonalization is discussed. In the following, we will outline the main idea and generalize it in order to cover the case of a non-symmetric $L$, i.e., the Oseen flow. The advantage of the approximation is that $C$ is a circulant matrix, which can be diagonalized with the Fourier matrix
\begin{align*}
    C = F^{-1} D F,
\end{align*}
where $F$ may be approximated using the forward Fourier transform, $F^{-1}$ with the inverse Fourier transform, and $D = \text{diag}(F c_1)$ with $c_1$ being the first column of $C$.
This approximation and the Kronecker product form allow us to diagonalize the whole system
\begin{align*}
   \mathcal{P}_C = \mathcal{F}^{-1} \underbrace{\left( \begin{pmatrix}
        \tau I_t & D^* \\
        D & -\frac{\tau}{\beta} I_t
    \end{pmatrix}\otimes
    \begin{pmatrix}
        M & 0 \\
        0 & 0
    \end{pmatrix}
    +
    \tau \begin{pmatrix}
        0 & 0 \\
        I_t & 0
    \end{pmatrix} \otimes 
    \begin{pmatrix}
        L & B^T \\
        B & 0
    \end{pmatrix}
    +
    \tau \begin{pmatrix}
        0 & I_t \\
        0 & 0
    \end{pmatrix} \otimes 
    \begin{pmatrix}
        L^T & B^T \\
        B & 0
    \end{pmatrix}
    \right)}_{=\,\bar{\mathcal{G}}} \mathcal{F},
\end{align*}
where $\mathcal{F}$ is the block Fourier transform in time
\begin{align*}
    \mathcal{F} = \begin{pmatrix}
        F & 0 \\
        0 & F
    \end{pmatrix} \otimes
    \begin{pmatrix}
        I_{n_v} & 0 \\
        0 & I_{n_p}
    \end{pmatrix}.
\end{align*}
The integers $n_v$ and $n_p$ are the numbers of degrees of freedom in the velocity and pressure approximation spaces. 
The operator $\mathcal{F}$ can be applied in $O((n_v + n_p) n_t \log n_t$) floating point operations. 

Upon permutation, $\bar{\mathcal{G}}$ takes a block diagonal form $\mathcal{G} = \operatorname{diag}(G_1, \dots, G_{n_t - 1})$ with the blocks
\begin{align}
    G_j = \begin{pmatrix}
        \tau M & d_j^* M + \tau L^T & 0 & \tau B^T \\
        d_j M + \tau L & -\frac{\tau}{\beta} M & \tau B^T & 0 \\
        0 & \tau B & 0 & 0 \\
        \tau B & 0 & 0 & 0
    \end{pmatrix},
    \label{eq:SubBlockG}
\end{align}
where $d_j$ is the $j$\ts{th} diagonal entry of $D$.
We can now solve for each block separately, i.e., in parallel. Each of the blocks can be interpreted as a version of a time-independent Stokes problem. Hence, we have arrived at a preconditioner that can be applied in a parallel-in-time fashion.

Since $\mathcal{P}_C$ represents a low-rank perturbation of $\mathcal{A}$, we can identify the majority of the eigenvalues of the preconditioned system.
\begin{theorem}
    The preconditioner $\mathcal{P}_C$ fulfills
\begin{align*}
\# \left\{ \mu \in \lambda\left( \mathcal{P}_C^{-1} \mathcal{A}\right) \mid \mu = 1 \right\} \ge 2 (n_t - 1) (n_v + n_p) - 2 n_v.
\end{align*}
\label{theorem:EigenvaluesPC}
\end{theorem}
Thus, the preconditioner has the desirable property of eigenvalue clustering. 
Nevertheless, the application of $\mathcal{P}_C$ still requires solving for a large number of stationary problems $G_j$ at each Krylov iteration, which can be expensive for even moderately sized problems. In that case, the overall preconditioner might become inefficient for its practical use. 
In the following, we will look at possible ways of solving the subblocks $G_j$ arising from the diagonalization. First, we will discuss possible preconditioners for the subsystem, then we will look at approaches for embedding these within the preconditioner $\mathcal{P}_C$.

\subsection{Preconditioning the subsystem}
\label{sec:SubSystemPreconditioners}
In \cite{Krendl2015EfficientEquations}, a preconditioner was developed for $\eqref{eq:SubBlockG}$. First, the block is decomposed into the form
\begin{align}
    G_j &= T_j^{(l)} Z_j T_j^{(r)},
    \label{eq:GJTransformation}
\end{align}
with the matrices
\begin{align*}
    T_j^{(l)} &= \begin{pmatrix}
        \sqrt{\tau} I & 0 & 0 & 0 \\
        \frac{d_{j, c}}{\sqrt{\tau}} i I & \frac{\sqrt{\tau}}{c_{j, 2}} I & 0 & 0 \\
        0 & 0 & \sqrt{\tau} c_{j, 2} I & 0 \\
        0 & 0 & \frac{d_{j, c} c_{j, 2}}{\sqrt{\tau}} i I & \sqrt{\tau} I
    \end{pmatrix}, \quad 
    T_j^{(r)} = \begin{pmatrix}
        \sqrt{\tau} I & -\frac{d_{j,c}}{\sqrt{\tau}} i I & 0 & 0 \\
        0 & \frac{\sqrt{\tau}}{c_{j, 2}} I & 0 & 0 \\
        0 & 0 & \sqrt{\tau} c_{j, 2} I & -\frac{d_{j, c} c_{j, 2}}{\sqrt{\tau}}i I \\
        0 & 0 & 0 & \sqrt{\tau} I
    \end{pmatrix}, \\
    Z_j &= \begin{pmatrix}
        M & c_{j, 1} M + c_{j, 2} L & 0 & B^T \\
        c_{j, 1} M + c_{j, 2} L & -M & B^T & 0 \\
        0 & B & 0 & 0 \\
        B & 0 & 0 & 0
    \end{pmatrix},
\end{align*}
and the constants
\begin{align*}
d_{j, r} = \operatorname{Re}(d_j), \quad & d_{j, c} = \operatorname{Im}(d_j), \\
c_{j, 1} = \frac{d_{j, r}}{\sqrt{\frac{\tau^2}{\beta} + d_{j, c}}}, \quad & c_{j, 2} = \frac{1}{\sqrt{\frac{1}{\beta} + \frac{d_{j, c}^2}{\tau^2}}}.
\end{align*}
Since $T_j^{(l)}$ and $T_j^{(r)}$ involve only block row manipulations and scalings, their inverses can be applied efficiently. This allows us to focus on preconditioning the matrix $Z_j$. 

The first attempt involves using
\begin{align}
    \widehat{P}_j = \begin{pmatrix}
        W_j & & & \\
        & W_j & & \\
        & & B W_j^{-1} B^T & \\
        & & & B W_j ^{-1} B^T
    \end{pmatrix}
    \label{eq:PJPreconditionerExact}
\end{align}
as a preconditioner for $Z_j$, where $W_j = M + c_{j, 1} M + c_{j, 2} L$.
In \cite{Krendl2015EfficientEquations}, it was shown that $\widehat{P}_j$ is symmetric positive definite and that the absolute values of the eigenvalues of the preconditioned matrix are contained in the following interval:
\begin{align}
    \left| \lambda(\widehat{P}_j^{-1} Z_j) \right| \subseteq \left[\frac{1}{\sqrt{12}}, \frac{1 + \sqrt{5}}{2}\right].
    \label{eq:SchurComplementHatBounds}
\end{align}
Even though $\widehat{P}_j$ leads to robust spectral properties for the preconditioned system and has a block diagonal form, the matrices $B W_j^{-1} B^T$ are still difficult to apply the inverses of, since in general they are dense and expensive to assemble. Thus, the authors of \cite{Krendl2015EfficientEquations} proposed an approximation based on the commutator argument
\begin{align*}
    B W_j^{-1} B^T \approx M_p \left(M_p + c_{j, 1} M_p + c_{j, 2} L_p\right)^{-1} K_p = \left(K_p^{-1} + c_{j, 1} K_p^{-1} + \nu c_{j, 2} M_p^{-1} \right)^{-1} = \widehat{S}_j.
\end{align*}
Note that $L_p = \nu K_p$, because we consider the Stokes problem. Plugging this into \eqref{eq:PJPreconditionerExact} results in a preconditioner denoted by $\widetilde{P}_j$. 

\subsection{Linear preconditioner}
The most immediate idea for a preconditioner of the overall system \eqref{eq:FinalA} is to start with $\mathcal{P}_C$ and replace all $Z_j$ by approximations. The structure of the resulting preconditioner, given some approximations $\bar{Z}_j$, is outlined in Algorithm~\ref{alg:LinearPrecond}. The preceding analysis suggests that $\widehat{P}_j$ and $\widetilde{P}_j$ form potentially good candidates. We will identify the preconditioners obtained from using these approximations at every time point by $\widehat{\mathcal{P}}_C$ and $\widetilde{\mathcal{P}}_C$. 

\begin{algorithm}[t]
\begin{algorithmic}
    \Function{$\bar{\mathcal{P}}_C^{-1}$}{$r$}

    \State $\hat{r} \gets \mathcal{F}r$
    \State Permute $\hat{r}$ to arrive at block-diagonal system
    \For{$j \in \{ 1, \dots, n_t - 1 \}$}
        \State $\hat{s}_j \gets \left(T_j^{(l)}\right)^{-1} \hat{r}_j$
        \State $\hat{x}_j \gets \bar{Z}_j^{-1} \hat{s}_j$
        \State $\hat{y}_j \gets \left(T_j^{(r)}\right)^{-1} \hat{x}_j$
    \EndFor
    \State Reverse permutation of $\hat{y}$
    \State $y \gets \mathcal{F}^{-1} \hat{y}$
    \State \Return $y$
    
    \EndFunction
\end{algorithmic}
\caption{Linear preconditioner $\bar{\mathcal{P}}_C$ for all-at-once system given an approximation $\bar{Z}_j$ of $Z_j$}
\label{alg:LinearPrecond}
\end{algorithm}

Similar to $\mathcal{P}_C$, we can derive estimates for the eigenvalues of the preconditioned systems $\widehat{\mathcal{P}}_C^{-1} \mathcal{A}$ and $\widetilde{\mathcal{P}}_C^{-1} \mathcal{A}$. The estimates are centered around the fact that $\mathcal{P}_C$ is a low-rank perturbation of the original system $\mathcal{A}$. More precisely, we have
\begin{align*}
    \mathcal{A} = \mathcal{P}_C + \mathcal{P}_R,
\end{align*}
where the matrix $\mathcal{P}_R$ has rank $2 n_v$. Because $\widehat{\mathcal{P}}_C$ is symmetric positive definite, we have that the eigenvalue analysis of the preconditioned system can be carried out on the following similar matrices:
\begin{align}
    \widehat{\mathcal{P}}_C^{-1} \mathcal{A} ~\sim~ \underbrace{\widehat{\mathcal{P}}_C^{-1/2} \mathcal{A} \widehat{\mathcal{P}}_C^{-1/2}}_{= \mathcal{H}} = \underbrace{\widehat{\mathcal{P}}_C^{-1/2} \mathcal{P}_C \widehat{\mathcal{P}}_C^{-1/2}}_{= \mathcal{K}} + \underbrace{\widehat{\mathcal{P}}_C^{-1/2} \mathcal{P}_R \widehat{\mathcal{P}}_C^{-1/2}}_{= \mathcal{L}}.
    \label{eq:PCHatSimilarMatrices}
\end{align}

Now, let the sequences $\eta_i$, $\kappa_i$, and $\lambda_i$ be the eigenvalues of the matrices $\mathcal{H}$, $\mathcal{K}$, and $\mathcal{L}$ in non-increasing order. Each eigenvalue occurs in the sequence according to its multiplicity. The matrix $\mathcal{K}$ has the same eigenvalues as $\widehat{\mathcal{P}}_C^{-1} \mathcal{P}_C$, whose eigenvalues are determined by the preconditioned subblocks, i.e.,
\begin{align*}
    \lambda\left(\widehat{\mathcal{P}}_C^{-1} \mathcal{P}_C\right) = \bigcup_j \lambda\left( \widehat{P}_j^{-1} Z_j \right).
\end{align*}
This allows us to establish bounds $\hat{a}, \hat{b} > 0$ such that
\begin{align*}
    \left|\lambda\left(\mathcal{K}\right)\right| = \left|\lambda\left(\widehat{\mathcal{P}}_C^{-1} \mathcal{P}_C\right)\right| \in [\hat{a}, \hat{b}].
\end{align*}
Due to the specific saddle-point structure of $\mathcal{P}_C$, Sylvester's law of inertia implies that the eigenvalues can be divided into two equally sized chunks of negative and positive eigenvalues, i.e., it holds that
\begin{align*}
    \kappa_i \in [\hat{a}, \hat{b}] \text{ for } i \in I^+_\kappa, \quad
    \kappa_i \in [-\hat{b}, -\hat{a}] \text{ for } i \in I^-_\kappa,
\end{align*}
with the index sets
\begin{align*}
    I^+_\kappa = \left\{ 1, \dots, \frac{N}{2} \right\}, \quad
    I^-_\kappa = \left\{\frac{N}{2} + 1, \dots, N\right\}.
\end{align*}
where $N = 2 (n_t - 1) (n_v + n_p)$. Due to the low-rank structure of $\mathcal{P}_R$, almost all eigenvalues $\lambda_i$ are zero. Specifically, there are $2 n_v$ non-zero eigenvalues. Even though we do not have bounds on the magnitude of these eigenvalues, we can use a similar procedure as before to show with Sylvester's law of inertia that
\begin{align*}
    \lambda_i \ge 0 \text{ for } i \in I^+_\lambda, \quad
    \lambda_i = 0 \text{ for } i \in I^0_\lambda, \quad
    \lambda_i \le 0 \text{ for } i \in I^-_\lambda,
\end{align*}
where the index sets are given by
\begin{align*}
    I^+_\lambda = \{ 1, \dots, n_v \}, \quad
    I^0_\lambda = \{n_v + 1, \dots, N - n_v\}, \quad
    I^-_\lambda = \{N - n_v + 1, \dots N\}.
\end{align*}

With that, we may derive estimates for $\eta_i$, which are broadly based on the following result from \cite{Weyl1912DasHohlraumstrahlung}:
\begin{lemma}
    Let $\bar{A}$ and $\bar{B}$ be arbitrary but fixed real symmetric matrices of size $n \times n$ and set $\bar{C} = \bar{A} + \bar{B}$. Denote the eigenvalues of $\bar{A}$, $\bar{B}$, and $\bar{C}$ by the sequences $\alpha_i$, $\beta_i$, and $\gamma_i$ in non-increasing order. Each eigenvalue occurs in the sequence according to its multiplicity. Then, it holds that
    \begin{align*}
        \gamma_{i + j - 1} \le \alpha_i + \beta_j
    \end{align*}
    for all $i + j - 1 \le n$.
    \label{theorem:SumEigenvalueBoundWeyl}
\end{lemma}
This allows us to carry over the initial estimate for $\mathcal{K}$ to a subset of the eigenvalues of $\mathcal{H}$.
\begin{theorem}
The preconditioner $\widehat{\mathcal{P}}_C$ fulfills
\begin{align*}
\# \left\{ \mu \in \lambda\left( \widehat{\mathcal{P}}_C^{-1} \mathcal{A}\right) \mid |\mu| \in [\hat{a}, \hat{b}] \right\} \ge 2 (n_t - 1) (n_v + n_p) - 4 n_v.
\end{align*}
\label{theorem:EigenvalueExactSchur}
\end{theorem}
\begin{proof}
Let the matrices $\mathcal{H}$, $\mathcal{K}$, and $\mathcal{L}$ and the corresponding eigenvalue sequences be defined as in \eqref{eq:PCHatSimilarMatrices}. First, we establish an estimate from above by $\hat{b}$ for some of the eigenvalues $\eta_k$ of $\widehat{\mathcal{P}}_C^{-1} \mathcal{A}$. The estimate of Lemma \ref{theorem:SumEigenvalueBoundWeyl}
\begin{align*}
    \eta_{i + j - 1} \le \kappa_i + \lambda_j \text{ for } i + j - 1 \le N
\end{align*}
simplifies to
\begin{align*}
    \eta_{i + n_v} \le \kappa_i \text{ for } i \le N - n_v
\end{align*}
by setting $j = n_v + 1$, since $n_v + 1 \in I_\lambda^0$. The estimate $\kappa_i \le \hat{b}$ holds for $i \in I_\kappa^+$. Thus, we want to restrict ourselves to the case $i \le \frac{N}{2}$. This allows us to infer that
\begin{align*}
    \eta_{i + n_v} \le \kappa_i \le \hat{b} \text{ for } i \in \left\{j \in I_\kappa^+ \mid j + n_v \le \frac{N}{2}\right\}.
\end{align*}
By substituting the index of $\eta$, this implies
\begin{align*}
    \eta_k \le \hat{b} \text{ for } k \in I_\eta^{\hat{b}} = \left\{n_v + 1, \dots, \frac{N}{2}\right\}.
\end{align*}

Next, we determine a set of eigenvalues for which the upper bound $-\hat{a}$ holds. The same theorem can be used for that, but instead of $i \in I_\kappa^+$, we are interested in indices $i \in I_\kappa^-$, since for this set we have the estimate $\kappa_i \le -\hat{a}$. Hence,
\begin{align*}
    \eta_{i + n_v} \le \kappa_i \le -\hat{a} \text{ for } i \in \left\{j \in I_\kappa^- \mid j + n_v \le N\right\},
\end{align*}
which leads to the estimate
\begin{align*}
    \eta_{k} \le -\hat{a} \text{ for } k \in I_\eta^{-\hat{a}} = \left\{ \frac{N}{2} + n_v + 1, \dots, N \right\}.
\end{align*}

The matrix $\mathcal{H}^- := -\mathcal{K} - \mathcal{L}$ has eigenvalues $\eta^-_k = -\eta_{N - k + 1}$. If we follow the same procedure as above for $\eta_k^-$, we receive the estimates
\begin{align*}
    \eta^-_k &\le \hat{b} \text{ for } k \in I_\eta^{\hat{b}}, \\
    \eta^-_k &\le -\hat{a} \text{ for } k \in I_\eta^{-\hat{a}},
\end{align*}
which in turn leads to
\begin{align*}
    \eta_k &\ge -\hat{b} \text{ for } k \in I_\eta^{-\hat{b}} = \left\{\frac{N}{2} + 1, \dots, N - n_v \right\}, \\
    \eta_k &\ge \hat{a} \text{ for } k \in I_\eta^{\hat{a}} = \left\{ 1, \dots, \frac{N}{2} - n_v \right\}.
\end{align*}

Hence,
\begin{alignat*}{3}
    \eta_k &\in [\hat{a}, \hat{b}] &&\text{ for } k \in I_\eta^+ = I_\eta^{\hat{a}} \cap I_\eta^{\hat{b}} = \left\{ n_v + 1, \dots, \frac{N}{2} - n_v \right\}, \\
    \eta_k &\in [-\hat{b}, -\hat{a}] &&\text{ for } k \in I_\eta^- = I_\eta^{-\hat{a}} \cap I_\eta^{-\hat{b}} = \left\{ \frac{N}{2} + n_v + 1, \dots, N - n_v \right\}.
\end{alignat*}
Finally, we get the result by determining the size of the set $| I_\eta^+ \cup I_\eta^-| = N - 4 n_v$.
\end{proof}

The same idea can be applied to the matrix $\widetilde{\mathcal{P}}_C^{-1} \mathcal{A}$ by first applying the following lemma: 
\begin{lemma}
    Assume that $\bar{A}$ and $\bar{B}$ are symmetric positive definite matrices and that $\bar{C}$ is a symmetric, invertible matrix with the following property:
    \begin{align*}
        \left| \lambda(\bar{A}^{-1} \bar{B}) \right| &\in [a, b], \\
        \left| \lambda(\bar{B}^{-1} \bar{C}) \right| &\in [c, d].
    \end{align*}
    Then, the product $\bar{A}^{-1} \bar{C}$ fulfills
    \begin{align*}
        \left| \lambda(\bar{A}^{-1} \bar{C}) \right| &\in [ac, bd].
    \end{align*}
    \label{theorem:Eigevals}
\end{lemma}
\begin{proof}
First, we consider the upper bound. It holds that
\begin{align*}
    \max \left| \lambda(\bar{A}^{-1} \bar{C}) \right| &= \max \left| \lambda(\bar{A}^{-1/2} \bar{C} \bar{A}^{-1/2}) \right| \\
    &= \max \left| \lambda(\bar{A}^{-1/2} \bar{B}^{1/2} \bar{B}^{-1/2} \bar{C} \bar{B}^{-1/2} \bar{B}^{1/2} \bar{A}^{-1/2}) \right| \\
    &= \left\| \bar{A}^{-1/2} \bar{B}^{1/2} \bar{B}^{-1/2} \bar{C} \bar{B}^{-1/2} \bar{B}^{1/2} \bar{A}^{-1/2} \right\|_2 \\
    &\le \left\| \bar{A}^{-1/2} \bar{B}^{1/2} \right\|_2 \left\| \bar{B}^{-1/2} \bar{C} \bar{B}^{-1/2} \right\|_2 \left\| \bar{B}^{1/2} \bar{A}^{-1/2} \right\|_2 \\
    &= \max \left| \lambda(\bar{A}^{-1} \bar{B}) \right| \max \left| \lambda(\bar{B}^{-1} \bar{C}) \right| \\
    &\le bd,
\end{align*}
where we used that
\begin{align*}
\max \left| \lambda(M) \right| = \| M\|_2
\end{align*}
for any symmetric matrix $M$. The lower bound follows from proceeding as above but with $(\bar{A}^{-1} \bar{C})^{-1}$ instead of $\bar{A}^{-1}\bar{C}$. Since
\begin{align*}
    \lambda((\bar{A}^{-1} \bar{C})^{-1}) = \{ \lambda^{-1} \mid \lambda \in \lambda(\bar{A}^{-1} \bar{C}) \},
\end{align*}
we get the upper bound
\begin{align*}
    \min \left| \lambda(\bar{A}^{-1} \bar{C}) \right| \ge a c.
\end{align*}~
\end{proof}
Since the commutator argument gives robust bounds $\widehat{S}_j^{-1} B W_j^{-1} B^T \in [c, d]$, it holds that
\begin{align*}
    \left|\lambda\left(\widetilde{\mathcal{P}}_C^{-1}\mathcal{P}_C\right)\right| \in [\hat{a}c, \hat{b}d] =: [\tilde{a}, \tilde{b}],
\end{align*}
which makes the following theorem a direct consequence using the same reasoning as in Theorem~\ref{theorem:EigenvalueExactSchur}:
\begin{theorem}
The preconditioner $\widetilde{\mathcal{P}}_C$ fulfills
\begin{align*}
\# \left\{ \mu \in \lambda\left( \widetilde{\mathcal{P}}_C^{-1} \mathcal{A}\right) \mid |\mu| \in [\tilde{a}, \tilde{b}] \right\} \ge 2 (n_t - 1) (n_v + n_p) - 4 n_v.
\end{align*}
\label{theorem:EigenvaluesInexactSchur}
\end{theorem}

Theorems \ref{theorem:EigenvalueExactSchur} and \ref{theorem:EigenvaluesInexactSchur} show that the vast majority of the eigenvalues of the relevant preconditioned systems are tightly contained in positive or negative clusters bounded away from the origin, with at most $4 n_v$ outlier eigenvalues due to the low-rank update used to enable a parallelizable-in-time block circulant approximation.

Theorems \ref{theorem:EigenvaluesPC}, \ref{theorem:EigenvalueExactSchur}, and \ref{theorem:EigenvaluesInexactSchur} can be verified in practice. Figure~\ref{fig:Eigenvalues} shows the real parts of the eigenvalues of some preconditioned systems. We restrict ourselves to rather small problems with $n_t = 10$ for this analysis due to the memory and computational requirements incurred by the computation of the eigenvalues. Figure~\ref{fig:RealPlots:P1} shows that the majority of the eigenvalues of $\mathcal{P}_C^{-1} \mathcal{A}$ is $1$ and the ratio of the remaining eigenvalues is less than $1/(n_t - 1)$. In Figure~\ref{fig:RealPlots:P2InvA}, we see that the majority of the eigenvalues of $\widehat{\mathcal{P}}_C^{-1}\mathcal{A}$ are within the theoretical bounds \eqref{eq:SchurComplementHatBounds}, which are depicted as the red dashed lines. Figure~\ref{fig:RealPlots:P3InvP1} verifies Theorem~\ref{theorem:Eigevals} in combination with the bounds arising from the application of the Schur complement approximation \eqref{eq:PJPreconditionerExact} and the commutator argument. Lastly, we can observe the implications of Theorem~\ref{theorem:EigenvaluesInexactSchur}, in which case we can see that some of the eigenvalues are pushed outside of the bounds. 
\begin{figure}
    \centering
    \begin{subfigure}[b]{0.49\textwidth}
    \includegraphics[width=\textwidth]{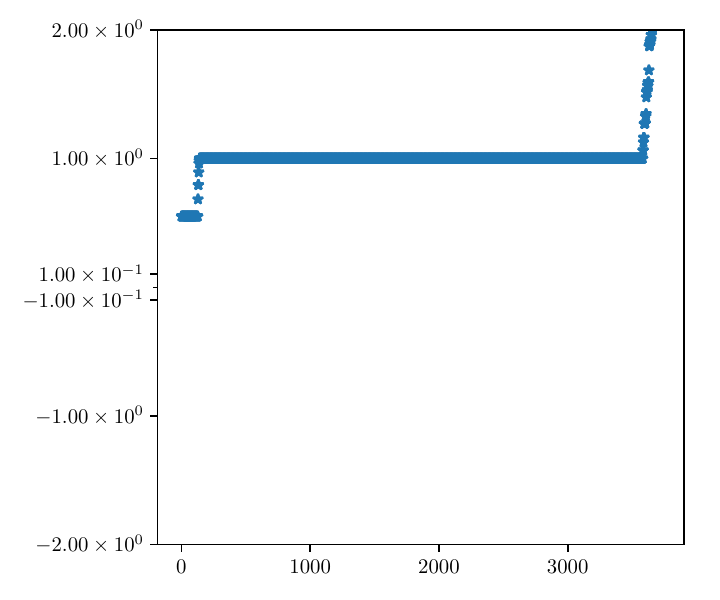}
    \caption{$\mathcal{P}_C^{-1} \mathcal{A}$}
    \label{fig:RealPlots:P1}
    \end{subfigure}
    \begin{subfigure}[b]{0.49\textwidth}
    \includegraphics[width=\textwidth]{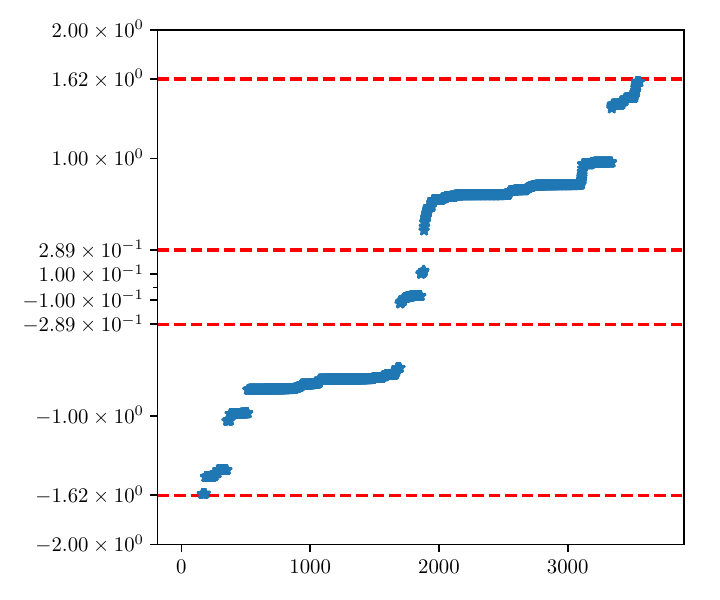}
    \caption{$\widehat{\mathcal{P}}_C^{-1} \mathcal{A}$}
    \label{fig:RealPlots:P2InvA}
    \end{subfigure}    
    \begin{subfigure}[b]{0.49\textwidth}
    \includegraphics[width=\textwidth]{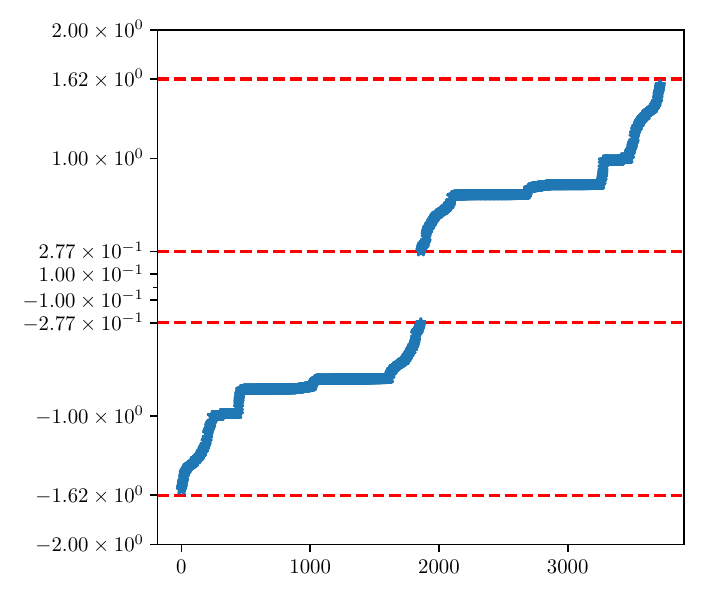}
    \caption{$\widetilde{\mathcal{P}}_C^{-1} \mathcal{P}_C$}
    \label{fig:RealPlots:P3InvP1}
    \end{subfigure}
    \begin{subfigure}[b]{0.49\textwidth}
    \includegraphics[width=\textwidth]{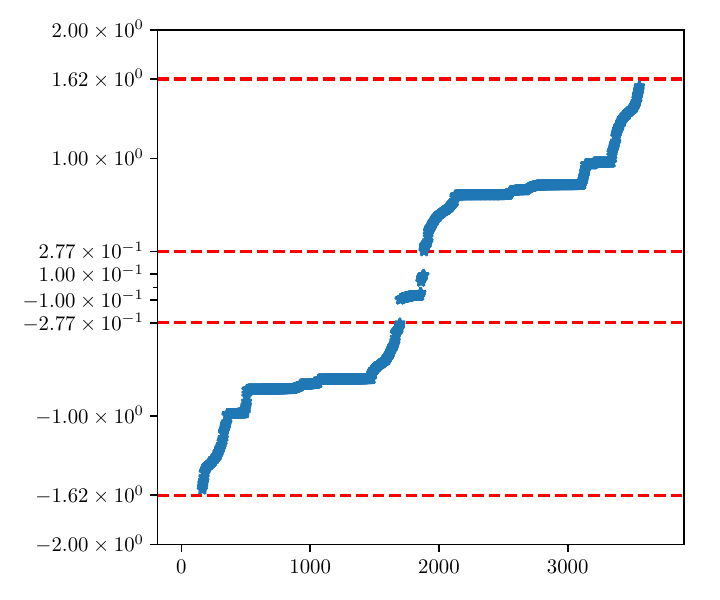}
    \caption{$\widetilde{\mathcal{P}}_C^{-1} \mathcal{A}$}
    \label{fig:RealPlots:P3InvA}
    \end{subfigure}
\caption{The (ordered) eigenvalues of the preconditioned matrices, with $n_t = 10$. The vertical axis represents the eigenvalues (or, in the top left plot, their real part) and the horizontal axis shows the index.}
\label{fig:Eigenvalues}
\end{figure}

The potential to exploit parallelizability of $\widehat{\mathcal{P}}_C$ and $\widetilde{\mathcal{P}}_C$ is constrained by how often the preconditioners have to be applied. This limitation arises from the need for communication between each application of the preconditioners. If we assume that a more accurate approximation of the subblocks leads to a reduction in the number of iterations, then there exists a trade-off between the quality of the approximation of the subblocks $G_j$ and the number of iterations required by the Krylov subspace solver.
As we will see in the experiments, the above preconditioners do not provide a sufficient level of accuracy. Thus, we will look at preconditioner designs which use the approximations described above to enhance the performance of an inner solver. 

\subsection{Nonlinear preconditioner}
Theorem~\ref{theorem:EigenvaluesInexactSchur} suggests that the performance of $\widetilde{\mathcal{P}}_C$ might be improved by narrowing the bounds $\tilde{a}$ and $\tilde{b}$, which is equivalent to solving for $G_j$ more accurately. Since we exclude the idea of using a direct solver due to its computational cost for very finely discretized problems, we propose to solve the subblocks to a certain relative tolerance only. As there is a good preconditioner $\widetilde{P}_j$ available for $G_j$, the subsystems can be solved with a preconditioned Krylov subspace method such as MINRES and GMRES, which will be called the {\it inner solver}. Accordingly, the tolerance $\varepsilon > 0$ of the subsystems will be called the {\it inner tolerance}. Since the preconditioner for the subsystems gives a good clustering of the eigenvalues, the inner solver will converge quickly and robustly with respect to the number of spatial degrees of freedom. This preconditioner will be denoted as $\widetilde{\mathcal{P}}_C^{(\text{NL})}$. 

We will call the iteration for solving the all-at-once system the {\it outer iteration}. With the linear preconditioner, we would propose using MINRES or GMRES for the outer iteration, but this is no longer possible in this case. Since the inner solver is a Krylov solver, it is nonlinear. GMRES, however, requires a linear, i.e., also constant, preconditioner. A method that can cope with alternating preconditioners in each iteration, is flexible GMRES (FGMRES), cf.~\cite{Saad1993AAlgorithm}. The pseudocode for the preconditioner is presented in Algorithm~\ref{alg:FGMRESDiag}.
\begin{algorithm}[t]
\begin{algorithmic}
    \Function{$\left(\widetilde{\mathcal{P}}_C^{(\text{NL})}\right)^{-1}$}{$r$}

    \State $\hat{r} \gets \mathcal{F}r$
    \State Permute $\hat{r}$ to arrive at block-diagonal system
    \For{$j \in \{ 1, \dots, n_t - 1 \}$}
        \State $\hat{s}_j \gets \left(T_j^{(l)}\right)^{-1} \hat{r}_j$
        \State $\hat{x}_j \gets \operatorname{GMRES}(Z_j, \widehat{P}_j, \hat{s}_j, \text{tol}=\varepsilon)$
        \State $\hat{y}_j \gets \left(T_j^{(r)}\right)^{-1} \hat{x}_j$
    \EndFor
    \State Reverse permutation of $\hat{y}$
    \State $y \gets \mathcal{F}^{-1} \hat{y}$
    \State \Return $y$
    
    \EndFunction
\end{algorithmic}
\caption{Nonlinear version of preconditioner for all-at-once system}
\label{alg:FGMRESDiag}
\end{algorithm}

It is important to remark that we should ideally not require many outer iterations, since the memory requirements of FGMRES grow linearly with the number of iterations. Let $N$ be the number of degrees of freedom for the all-at-once optimal control problem. Then, with $m$ denoting the number of outer iterations, we need to allocate memory in the order of $O(Nm)$. Considering problems with possibly billions of degrees of freedom, it becomes clear that this reaches the limits of contemporary computing facilities rather quickly. Intuitively, we would expect that if the solution of the subsystems is better, i.e., if $\varepsilon$ is smaller, fewer outer iterations are required. This would give us better control of the aforementioned bottleneck.

The nonlinear preconditioner requires one to apply the inverse of several finite element matrices. Since the exact solution of systems with these matrices can be expensive and can often be approximated without significant loss in the preconditioner's efficiency, we choose to use approximate solutions. One of the key methods here is the Chebyshev semi-iteration (cf.~\cite{Golub1961ChebyshevI}), which shows fast convergence for mass matrices due to existing robust spectral estimates for the mass matrix, see e.g., \cite{Wathen2008ChebyshevMatrix} and \cite{Wathen1987RealisticMatrix}. The other method utilized is a geometric multigrid method (MG, see \cite{Briggs2000AEdition,Hackbusch1985Multi-GridApplications}), of which various modes exist. In this work, we choose to apply such a multigrid method with successive over-relaxation for the smoothing.
An overview of where these methods are applied within the preconditioner can be found in Figure~\ref{fig:StokesSolversOverview}.
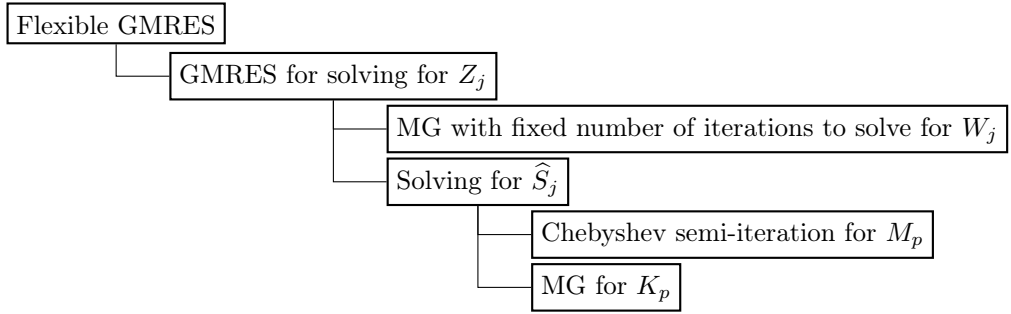
\begin{figure}
\begin{tikzpicture}[dirtree]
  \node [text height=0.8em, minimum height=1.5em] {Flexible GMRES}
    child { node {GMRES for solving for $Z_j$}
        child { node {MG with fixed number of iterations to solve for $W_j$}}
        child { node {Solving for $\widehat{S}_j$}
            child { node {Chebyshev semi-iteration for $M_p$}}
            child { node {MG for $K_p$}}
        }
    };
\end{tikzpicture}

\caption{Overview of the solvers used for the application of $\widetilde{\mathcal{P}}_C^{(\text{NL})}$. We denote the use of multigrid with `MG'.}
\label{fig:StokesSolversOverview}
\end{figure}

\section{Preconditioner for the Oseen control problem}
\label{seq:PreconditionerOseen}
Additionally to the Stokes problem, we would like to consider the more general Oseen flow control problem, i.e., the case when $w \neq 0$. Since we allowed $L$ to be non-symmetric in Section~\ref{sec:DiagonalizationFFT}, the idea of the circulant approximation of the system and its subsequent diagonalization can be transferred to the Oseen problem without any adaptations. In practice, preconditioners based on the idea of Algorithm~\ref{alg:FGMRESDiag} show promising results. We will apply the same technique to this problem and, thus, we are interested in developing an efficient preconditioner for $G_j$. 
Unfortunately, the transformation described by \eqref{eq:GJTransformation} does not apply to the Oseen problem, since it relies on the symmetry of $L$. This leaves us with developing preconditioners for $G_j$ directly rather than for the simpler matrix $Z_j$.

First, we will partition $G_j$ into a 2-by-2 block form, look at each of the blocks separately, and discuss their approximations. We define the (1,1)-block as
\begin{align*}
    G_j^{(1,1)} = \begin{pmatrix}
        \tau M & d_j^* M + \tau L^T \\
        d_j M + \tau L & -\frac{\tau}{\beta} M\\
    \end{pmatrix},
\end{align*}
which is a saddle-point system. The special form of  $G_j^{(1,1)}$ allows us to use \eqref{eq:SchurComplementPearsonWathen} in order to achieve the following robust preconditioner:
\begin{align*}
    P_j^{(1,1)} = \begin{pmatrix}
        \tau M & 0 \\
        d_j M + \tau L & \frac{1}{\tau} Q_j M^{-1} Q_j^H
    \end{pmatrix}, \quad Q_j = d_j M + \tau L + \frac{\tau}{\sqrt{\beta}} M,
\end{align*}
where $Q_j^H$ is the conjugate transpose of $Q_j$. Although the eigenvalue bounds for \eqref{eq:SchurComplementPearsonWathen} were originally established for real matrices only, it can be shown that these bounds also hold for the above complex matrices.
The Schur complement of \eqref{eq:SubBlockG} has the form
\begin{align*}
    S_j = \begin{pmatrix}
        0 & \tau B \\
        \tau B & 0
    \end{pmatrix}
    \begin{pmatrix}
        \tau M & d_j^* M + \tau L^T \\
d_j M + \tau L & -\frac{\tau}{\beta} M
    \end{pmatrix}^{-1}
    \begin{pmatrix}
        0 & \tau B^T \\
        \tau B^T & 0
    \end{pmatrix}.
\end{align*}
Leveraging the commutator argument, we derive the following approximation:
\begin{align*}
    S_j &\approx \widehat{S}_j = \begin{pmatrix}
        0 & \tau M_p \\
        \tau M_p & 0
    \end{pmatrix}
    \begin{pmatrix}
        \tau M_p & d_j^* M_p + \tau L_p^T \\
d_j M_p + \tau L_p & -\frac{\tau}{\beta} M_p
    \end{pmatrix}^{-1}
    \begin{pmatrix}
        0 & \tau K_p \\
        \tau K_p & 0
    \end{pmatrix}.
\end{align*}

Now, we have to put this together to efficiently precondition the overall system $G_j$. One idea would be to use the preconditioner \eqref{eq:SaddlePointTriangularPreconditioner} with the above approximations, which leads to
\begin{align}
\widetilde{P}_j =
    \begin{pmatrix}
        P_j^{(1,1)} & 0 \\
        G_j^{(2, 1)} & \widehat{S}_j
    \end{pmatrix}.
    \label{eq:OseenPreconditionerFirstIdea}
\end{align}
In practice, however, this preconditioner does not perform well if the approximation of the (1,1)-block is not accurate enough. Indeed, if $P_j^{(1,1)}$ in \eqref{eq:OseenPreconditionerFirstIdea} is replaced by its exact counterpart $G_j^{(1,1)}$ the performance was acceptable in experiments. This lets us conclude that a key contributor to the failure of $P_j$ is that the approximation of the (1,1)-block does not suffice for a good overall performance. We, therefore, try to improve the approximation of this block. 

One could follow the preceding procedure and nest a nonlinear Krylov subspace solver for the (1,1)-block in \eqref{eq:OseenPreconditionerFirstIdea}. Nevertheless, this would require us to use FGMRES instead of GMRES in Algorithm~\ref{alg:FGMRESDiag} and, therefore, lead to a threefold nested application of Krylov subspace solvers. The resulting additional complexity and the limited convergence theory of FGMRES gives reason to instead develop a linear approximation of $G_j^{(1,1)}$. A good candidate for that is the Uzawa iteration with a fixed number of iterations. More precisely, we use the inexact preconditioned Uzawa method or preconditioned Arrow--Hurwicz method, which is given by the following iterative form:
\begin{align}
\begin{aligned}
    x_1^{(k + 1)} &= x_1^{(k)} + \frac{1}{\tau}\widehat{M}^{-1} \left(b_1 - \tau M x_1^{(k)} - (d_j M + \tau L)^H x_2^{(k)} \right), \\
    x_2^{(k + 1)} &= x_2^{(k)} - \frac{1}{\mu} \widehat{S}^{-1} \left(b_2 - (d_j M + \tau L) x_1^{(k + 1)} + \frac{\tau}{\beta} M x_2^{(k)} \right),
\end{aligned}
\label{eq:Uzawa11}
\end{align}
for the solution of a saddle-point system of the form
\begin{align*}
    G_j^{(1,1)}
    \begin{pmatrix}
        x_1 \\ x_2
    \end{pmatrix}
    =
    \begin{pmatrix}
        b_1 \\ b_2
    \end{pmatrix},
\end{align*}
where $\widehat{M}$ is an approximation of the mass matrix and $\mu > 0$ is a parameter of the method. For a detailed discussion of the Uzawa method, we refer to \cite{Rees2010PreconditioningOptimization}. A suitable parameter choice for $\mu$ is 
\begin{align*}
    \mu = \frac{\lambda_{\min}(\widehat{S}^{-1}S) + \lambda_{\max}(\widehat{S}^{-1} S)}{2}.
\end{align*}
In our case, the minimum and maximum eigenvalues of $\widehat{S}^{-1} S$ are $1/2$ and $1$. Thus, we choose $\mu = 3/4$.
The resulting preconditioner is denoted by $\widetilde{P}_j^{(\text{UZ})}$ and the computation of its inverse is summarized in Algorithm~\ref{alg:UzawaGj}. The overall nonlinear preconditioner will be denoted by $\widetilde{\mathcal{P}}_C^{(\text{UZ})}$.
\begin{algorithm}[t]
\begin{algorithmic}
    \Function{$\left(\widetilde{P}_j^{(\text{UZ})}\right)^{-1}$}{$r$}
    \State $\begin{pmatrix}r_1^T & r_2^T & r_3^T & r_4^T\end{pmatrix}^T \gets r$
    \State $\begin{pmatrix}y_1 \\ y_2\end{pmatrix} \gets \begin{pmatrix}x_1^{(k)}\\ x_2^{(k)}\end{pmatrix}$ according to \eqref{eq:Uzawa11} with $x_1^{(0)} = x_2^{(0)} = 0$, $b_1 = r_1$, and $b_2 = r_2$
    \State $\begin{pmatrix}y_3 \\ y_4\end{pmatrix} \gets -\widehat{S}_j^{-1} \left(\begin{pmatrix} r_3 \\ r_4 \end{pmatrix} - \tau \begin{pmatrix}B y_2 \\ B y_1\end{pmatrix}\right)$
    \State \Return $\begin{pmatrix}
        y_1^T & y_2^T & y_3^T & y_4^T
    \end{pmatrix}^T$
    
    \EndFunction
\end{algorithmic}
\caption{Computation of the preconditioner $\widetilde{P}_j^{(\text{UZ})}$}
\label{alg:UzawaGj}
\end{algorithm}

Similarly to the Stokes control problem, the final form of the preconditioner does not solve for the finite element matrices exactly but uses approximations instead. Figure~\ref{fig:OseenSolversOverview} depicts the structure of the algorithm and where these approximations are used.
\begin{figure}
\begin{tikzpicture}[dirtree]
  \node [text height=0.8em, minimum height=1.5em] {Flexible GMRES}
    child { node {GMRES for solving for $G_j$}
        child { node {Uzawa for $G_j^{(1,1)}$}
            child { node {Chebyshev semi-iteration for $M$}}
            child { node {MG for $Q_j$}}
        }
        child { node {Solving for $\widehat{S}_j$}
            child { node {Chebyshev semi-iteration for $M_p$}}
            child { node {MG for $K_p$}}
        }
    };
\end{tikzpicture}

\caption{Overview of the solvers used for the application of $\widetilde{\mathcal{P}}_C^{(\text{UZ})}$.}
\label{fig:OseenSolversOverview}
\end{figure}
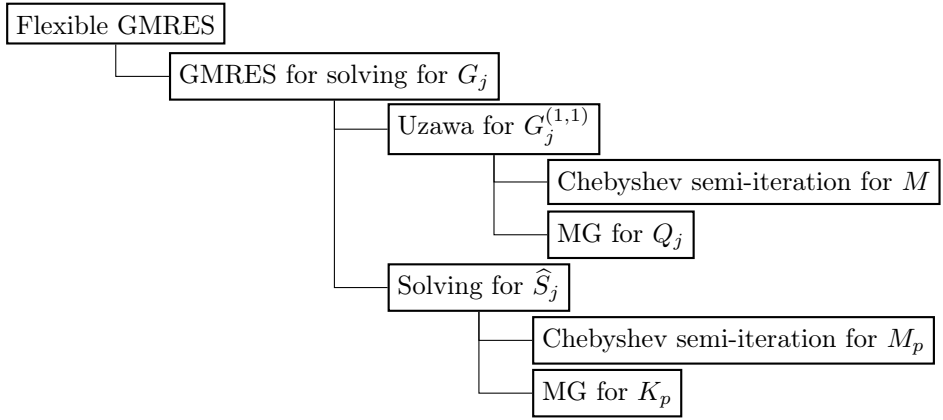

\section{Numerical experiments}
\label{seq:NumericalExperiments}
The performance of the preconditioners motivated above will now be showcased with numerical experiments. We will first analyze the robustness of the preconditioner with respect to some model parameters and the number of timesteps, followed by a presentation of the potential computational speedup through parallelization. The experiments are conducted for two-dimensional Stokes and Oseen problems on the domain $\Omega = [-1, 1]^2$, where $P_2$-$P_1$-elements are used for its approximation. We restrict ourselves to two levels of spatial resolution, the coarse one having a maximum element size of $h_c = 1 / 32$ and the finer one having $h_f = 1 / 64$.  

Within our implementation, we utilize the high-level finite element library DOLFINx \cite{Baratta2023DOLFINx:Environment} for spatial discretization. For miscellaneous numerical linear algebra, operations such as linear solvers, the libraries SciPy \cite{Virtanen2020SciPyPython} and petsc4py \cite{Dalcin2011ParallelPython} are used. 
Additionally to the functionalities for parallel computations shipped with petsc4py, the library mpi4py \cite{Dalcin2021Mpi4py:Development} is used for necessary nonstandard communication within parallel infrastructures. The experiments are run on varying numbers of processors of the model Intel(R) Xeon(R) CPU E7-4820 v2 @ 2.00GHz. 

\subsection{Stokes problem}
For the verification of the preconditioners discussed in Section~\ref{seq:PreconditionerStokes}, we consider the problem used in \cite{Leveque2022Parameter-robustControl}. The exact solution of the problem is known and it is defined by the following desired velocity state, forcing term, and exact solution when $\nu = 1$:
\begin{alignat*}{3}
v_d(x, t) &= 4 \beta \left[x_2 \left(2 \left(3 x_1^2 - 1\right)\left(x_2^2 - 1\right) + 3 \left(x_1^2 - 1\right)^2\right), -x_1 \left(3 \left(x_2^2 - 1\right)^2 + 2 \left(x_1^2 - 1\right)\left(3 x_2^2 - 1\right)\right)\right]^T \span \span \\
&\quad + e^{T - t} &&\left[20 x_1 x_2^3 + 2 \beta x_2 \left(\left(x_1^2 - 1\right)^2 \left(x_2^2 - 7\right) - 4 \left(3 x_1^2 - 1\right) \left(x_2^2 - 1\right) + 2\right) \right., \\
&\quad\quad&&\left. 5 \left(x_1^4 - x_2^4\right) - 2 \beta x_1 \left(\left(x_2^2 - 1\right)^2 \left(x_1^2 - 7\right) - 4 \left(x_1^2 - 1\right)\left(3 x_2^2 - 1\right) - 2\right)\right]^T, \\
f(x, t) &= e^{T - t} \left[-20 x_1 x_2^3 - 2 x_2 \left(x_1^2 - 1\right)^2 \left(x_2^2 - 1\right), 
5 \left(x_2^4 - x_1^4\right) + 2 x_1 \left(x_1^2 - 1\right) \left(x_2^2 - 1\right)^2\right]^T  \span \span\\
&\quad +\left[2 x_2 \left(x_1^2 - 1\right)^2 \left(x_2^2 - 1\right), -2 x_1 \left(x_1^2 - 1\right) \left(x_2^2 - 1\right)^2\right]^T,  \span \span\\
v(x, t) &= e^{T - t} \left[20 x_1 x_2^3, 5 x_1^4 - 5 x_2^4\right]^T,  \span \span \\
p(x, t) &= e^{T - t} \left(60 x_1^2 x_2 - 20 x_2^3\right) + \text{constant}, \span \span
\end{alignat*}
where the initial and boundary conditions are given implicitly by $v(x, t)$. Moreover, we used this solution to verify the correctness and convergence of the obtained discretized solutions. For our subsequent experiments, we retain the above choices of $v_d(x,t)$ and $f(x,t)$ for the desired state and forcing term.

\subsubsection{Linear preconditioner}
We will first examine the linear preconditioner $\widetilde{\mathcal{P}}_C$, which is applied within a GMRES iterative solver. The iterative solver is configured to restart every $30$ iterations. For this and all subsequent experiments, the viscosity is set to $\nu = 1\mathrm{e}{-2}$, in each application of MG $4$ V-cycles are used, and the number of Chebyshev semi-iterations for approximately applying the inverse of the mass matrix is set to $10$. Unless otherwise stated, the final time is set to $T=10$. Table~\ref{tab:StokesLinear} shows the numbers of GMRES iterations and CPU times (in seconds), for various numbers of timesteps and different values of the regularization parameter~$\beta$. In this setting, all tests were run sequentially without leveraging the parallelizability in time. For certain large numbers of timesteps, the effective runtime would have exceeded the available resources. These experiments were computed in parallel and the runtime is reported as ``$*$''. The listed number of degrees of freedom (DOFs) includes all variables for the all-at-once formulation in space and time. 

We can observe that the number of GMRES iterations grows significantly with the problem size for $\beta=1\mathrm{e}{-1}$, which is not the behaviour we are looking for in a preconditioner. For the other two values of $\beta$, the number of iterations remains in the same order independent of the problem size, which indicates that in this scenario $\widetilde{\mathcal{P}}_C$ shows some robustness concerning the number of timesteps. Furthermore, the computations are done sequentially, which means the runtime is expected to scale linearly with the overall problem size and the number of iterations of the solver. This aligns with the obtained data. The large number of iterations and the long runtime therefore incurred show the limitations in the applicability of this preconditioner.

\begin{table}
\centering
\small
\begin{tabular}{|r|r|r|r|r|r|r|r||r|r|r|r|r|r|r|r|r|r|r|r|r|}
\hline
 & \multicolumn{7}{c||}{$h=h_c$} & \multicolumn{7}{c|}{$h=h_f$} \\ \hline
    \backslashbox{$n_t$}{$\beta$} &  & \multicolumn{2}{c|}{$1\mathrm{e}{-1}$} & \multicolumn{2}{c|}{$1\mathrm{e}{-3}$}  & \multicolumn{2}{c||}{$1\mathrm{e}{-4}$} &  & \multicolumn{2}{c|}{$1\mathrm{e}{-1}$} & \multicolumn{2}{c|}{$1\mathrm{e}{-3}$}  & \multicolumn{2}{c|}{$1\mathrm{e}{-4}$}  \\ \hline
       & \#DOFs & \rotatebox{90}{GMRES~} & \rotatebox{90}{CPU [$\mathrm{s}$]}  & \rotatebox{90}{GMRES~} & \rotatebox{90}{CPU [$\mathrm{s}$]} & \rotatebox{90}{GMRES~} & \rotatebox{90}{CPU [$\mathrm{s}$]} & \#DOFs & \rotatebox{90}{GMRES~} & \rotatebox{90}{CPU [$\mathrm{s}$]}  & \rotatebox{90}{GMRES~} & \rotatebox{90}{CPU [$\mathrm{s}$]} & \rotatebox{90}{GMRES~} & \rotatebox{90}{CPU [$\mathrm{s}$]} \\ \hline
$15$ & $2.86\mathrm{e}{5}$& $59$ & $42$& $50$ & $39$& $42$ & $29$ & $1.13\mathrm{e}{6}$& $68$ & $208$& $52$ & $162$& $48$ & $186$\\ \hline 
$31$ & $5.91\mathrm{e}{5}$& $68$ & $112$& $50$ & $95$& $43$ & $62$ & $2.33\mathrm{e}{6}$& $72$ & $508$& $52$ & $358$& $48$ & $316$\\ \hline 
$63$ & $1.20\mathrm{e}{6}$& $84$ & $277$& $50$ & $163$& $42$ & $140$ & $4.73\mathrm{e}{6}$& $90$ & $1245$& $52$ & $732$& $48$ & $601$\\ \hline 
$127$ & $2.42\mathrm{e}{6}$& $134$ & $1049$& $50$ & $320$& $44$ & $272$ & $9.53\mathrm{e}{6}$& $137$ & $3788$& $54$ & $1408$& $48$ & $1347$\\ \hline 
$255$ & $4.86\mathrm{e}{6}$& $252$ & $2601$& $52$ & $570$& $44$ & $525$ & $1.91\mathrm{e}{7}$& $258$ & $13440$& $54$ & $2608$& $48$ & $2676$\\ \hline 
$511$ & $9.75\mathrm{e}{6}$& $522$ & $*$& $54$ & $*$& $44$ & $*$ & $3.83\mathrm{e}{7}$& $796$ & $*$& $66$ & $*$& $48$ & $*$\\ \hline 
$1023$ & $1.95\mathrm{e}{7}$& $1277$ & $*$& $81$ & $*$& $46$ & $*$ & $7.67\mathrm{e}{7}$& $*$ & $*$& $*$ & $*$& $*$ & $*$\\ \hline 
\end{tabular}
\caption{Solving the Stokes problem with $\widetilde{\mathcal{P}}_C$ for different parameter settings.}
\label{tab:StokesLinear}
\end{table}

\subsubsection{Nonlinear preconditioner}
\label{seq:NumericalExperimentsStokesNonlinear}
Next, we will consider the nonlinear preconditioner $\widetilde{\mathcal{P}}^{(\text{NL})}_C$. As already mentioned, we have to use FGMRES as an outer iterative solver instead of GMRES. Given the substantial size of the problem and the increased memory requirements of FGMRES, we are required to use small numbers of iterations between each restart. Experiments suggest that $10$ iterations are sufficient to ensure convergence for the considered problems. As an inner solver, we use GMRES. The inner tolerance is set to $\varepsilon = 1\mathrm{e}{-2}$, which was found to be sufficient in numerical experiments, meaning that a smaller tolerance did not lead to a noticeable reduction of the outer iterations. All other parameters are chosen identically to those used in the experiments for the linear preconditioner. Tables~\ref{tab:StokesLevel4} and \ref{tab:StokesLevel5} show convergence results for the two different levels of spatial discretization, specifically the number of outer FGMRES iterations, the average number of inner GMRES iterations (to the nearest integer), and the CPU time required. We are able to observe robustness of the number of outer iterations. Additionally, the number of inner iterations stays within the same order independent of the parameters, which ensures a linear scaling of the runtime in $n_t$. Additional experiments showed that the preconditioner can deal with smaller regimes of $\nu$ robustly. However, the MG iterations do not necessarily guarantee convergence for $\nu < 1\mathrm{e}{-3}$, which is currently the bottleneck of the solver.
\begin{table}
\centering
\small
\begin{tabular}{|r|r|r|r|r|r|r|r|r|r|r|}
\hline
    \backslashbox{$n_t$}{$\beta$} &  & \multicolumn{3}{c|}{$1\mathrm{e}{-1}$} & \multicolumn{3}{c|}{$1\mathrm{e}{-3}$}  & \multicolumn{3}{c|}{$1\mathrm{e}{-4}$}  \\ \hline
       & \#DOFs & \rotatebox{90}{FGMRES~} & \rotatebox{90}{av.\ GMRES~} & \rotatebox{90}{CPU [$\mathrm{s}$]}  & \rotatebox{90}{FGMRES~} & \rotatebox{90}{av.\ GMRES~} & \rotatebox{90}{CPU [$\mathrm{s}$]} & \rotatebox{90}{FGMRES~} & \rotatebox{90}{av.\ GMRES~} & \rotatebox{90}{CPU [$\mathrm{s}$]} \\ \hline
$15$ & $2.86\mathrm{e}{5}$& $4$ & $32$ & $59$& $3$ & $26$ & $49$& $3$ & $33$ & $52$\\ \hline 
$31$ & $5.91\mathrm{e}{5}$& $5$ & $34$ & $226$& $4$ & $26$ & $127$& $3$ & $26$ & $96$\\ \hline 
$63$ & $1.20\mathrm{e}{6}$& $5$ & $29$ & $276$& $4$ & $23$ & $173$& $4$ & $25$ & $239$\\ \hline 
$127$ & $2.42\mathrm{e}{6}$& $6$ & $29$ & $932$& $4$ & $21$ & $446$& $4$ & $21$ & $395$\\ \hline 
$255$ & $4.86\mathrm{e}{6}$& $6$ & $26$ & $1654$& $4$ & $18$ & $571$& $4$ & $18$ & $727$\\ \hline 
$511$ & $9.75\mathrm{e}{6}$& $7$ & $27$ & $*$& $4$ & $17$ & $*$& $4$ & $15$ & $*$\\ \hline 
$1023$ & $1.95\mathrm{e}{7}$& $7$ & $26$ & $*$& $5$ & $23$ & $*$& $5$ & $21$ & $*$\\ \hline 
\end{tabular}
\caption{Solving the Stokes problem with $\widetilde{\mathcal{P}}_C^{(\text{NL})}$ for $h = h_c$.}
\label{tab:StokesLevel4}
\end{table}

\begin{table}
\centering
\small
\begin{tabular}{|r|r|r|r|r|r|r|r|r|r|r|}
\hline
    \backslashbox{$n_t$}{$\beta$} &  & \multicolumn{3}{c|}{$1\mathrm{e}{-1}$} & \multicolumn{3}{c|}{$1\mathrm{e}{-3}$}  & \multicolumn{3}{c|}{$1\mathrm{e}{-4}$}  \\ \hline
       & \#DOFs & \rotatebox{90}{FGMRES~} & \rotatebox{90}{av.\ GMRES~} & \rotatebox{90}{CPU [$\mathrm{s}$]}  & \rotatebox{90}{FGMRES~} & \rotatebox{90}{av.\ GMRES~} & \rotatebox{90}{CPU [$\mathrm{s}$]} & \rotatebox{90}{FGMRES~} & \rotatebox{90}{av.\ GMRES~} & \rotatebox{90}{CPU [$\mathrm{s}$]} \\ \hline
$15$ & $1.13\mathrm{e}{6}$& $4$ & $34$ & $360$& $3$ & $35$ & $252$& $3$ & $39$ & $301$\\ \hline 
$31$ & $2.33\mathrm{e}{6}$& $4$ & $32$ & $764$& $4$ & $28$ & $641$& $3$ & $31$ & $474$\\ \hline 
$63$ & $4.73\mathrm{e}{6}$& $4$ & $30$ & $1263$& $4$ & $27$ & $890$& $4$ & $27$ & $1140$\\ \hline 
$127$ & $9.53\mathrm{e}{6}$& $5$ & $33$ & $3762$& $4$ & $23$ & $1735$& $4$ & $25$ & $2247$\\ \hline 
$255$ & $1.91\mathrm{e}{7}$& $6$ & $33$ & $8483$& $4$ & $21$ & $3631$& $4$ & $22$ & $3324$\\ \hline 
$511$ & $3.83\mathrm{e}{7}$& $6$ & $30$ & $*$& $4$ & $17$ & $*$& $4$ & $19$ & $*$\\ \hline 
$1023$ & $7.67\mathrm{e}{7}$& $6$ & $27$ & $*$& $4$ & $16$ & $*$& $4$ & $16$ & $*$\\ \hline 
\end{tabular}
\caption{Solving the Stokes problem with $\widetilde{\mathcal{P}}_C^{(\text{NL})}$ for $h = h_f$.}
\label{tab:StokesLevel5}
\end{table}

\subsection{Oseen problem}
Next, we will analyze the performance of the preconditioner for the Oseen problem introduced in Section~\ref{seq:PreconditionerOseen}. The test problem is inspired by the test problems for the Navier--Stokes equations in \cite{Leveque2022Parameter-robustControl} and are given by the following data:
\begingroup
\allowdisplaybreaks
\begin{align*}
    v_d(x, t) &= \left[ 2 x_2 \left(1 - x_1^4\right)\left(1 - x_2^4\right), -2 x_1 \left(1 - x_1^4\right)\left(1 - x_2^4\right) \right]^T \\
    &\quad + \begin{cases}
    \left[ 5 x_2 - 4, 0 \right]^T &\text{ if } x_2 \ge \frac{4}{5}, \\
    \left[0, 0\right]^T &\text{ otherwise},
    \end{cases} \\
    f(x, t) &= \left[ -20 x_1 x_2^3 - 2 x_2 \left(x_1^2 - 1\right)^2 \left(x_2^2 - 1\right), 5 \left(x_2^4 - x_1^4\right) + 2 x_1 \left(x_1^2 - 1\right) \left(x_2^2 - 1\right)^2 \right]^T, \\
    v_0(x) &= \begin{cases}
        [1, 0]^T &\text{ if } x_1 > -x_2 \text{ and } x_1 < x_2, \\
        [0, 0]^T &\text{ otherwise},
    \end{cases} \\
    h(x, t) &= \begin{cases}
        [1, 0]^T &\text{ if } x_2 = 1, \\
        [0, 0]^T &\text{ otherwise},
    \end{cases} \\
    w(x) &= \begin{cases}
        c_1(x) \left[\left(\frac{100}{99}\right)^2 x_2, -\left(\frac{100}{49}\right)^2 \left(x_1 - \frac{1}{2}\right)\right]^T &\text{ if } c_1(x) \ge 0, \\
        c_2(x) \left[\left(-\frac{100}{99}\right)^2 x_2, \left(\frac{100}{49}\right)^2 \left(x_1 - \frac{1}{2}\right)\right]^T &\text{ if } c_2(x) \ge 0, \\
        [0, 0]^T &\text{ otherwise},
    \end{cases} \\
    c_1(x) &= 1 - \sqrt{\left(\frac{100}{49} \left(x_1 - \frac{1}{2}\right)\right)^2 + \left(\frac{100}{99} x_2\right)^2},\quad
    c_2(x) = 1 - \sqrt{\left(\frac{100}{49} \left(x_1 + \frac{1}{2}\right)\right)^2 + \left(\frac{100}{99} x_2\right)^2}.
\end{align*}
\endgroup
This is a lid-driven cavity problem, where the wind has two vortices and the desired state describes a velocity field with one vortex.

\subsubsection{Nonlinear preconditioner}
The preconditioner $\widetilde{\mathcal{P}}_C^{(\text{UZ})}$ introduced in Section~\ref{seq:PreconditionerOseen} is applied within the same iterative solver and associated settings as in Section~\ref{seq:NumericalExperimentsStokesNonlinear}. The number of iterations for the inner Uzawa method is set to $6$. Tables~\ref{tab:OseenLevel4} and \ref{tab:OseenLevel5} show a summary of the numerical results. Also for the Oseen problem, we observe robustness of the number of outer iterations as well as of the number of inner iterations. Accordingly, the total runtime scales linearly in $n_t$. Furthermore, convergence of the inner solver is achieved within few iterations, which suggests that $\widetilde{P}_j^{(\text{UZ})}$ is a good candidate for preconditioning the subblocks.

\begin{table}
\centering
\small
\begin{tabular}{|r|r|r|r|r|r|r|r|r|r|r|}
\hline
    \backslashbox{$n_t$}{$\beta$} &  & \multicolumn{3}{c|}{$1\mathrm{e}{-1}$} & \multicolumn{3}{c|}{$1\mathrm{e}{-3}$}  & \multicolumn{3}{c|}{$1\mathrm{e}{-4}$}  \\ \hline
       & \#DOFs & \rotatebox{90}{FGMRES~} & \rotatebox{90}{av.\ GMRES~} & \rotatebox{90}{CPU [$\mathrm{s}$]}  & \rotatebox{90}{FGMRES~} & \rotatebox{90}{av.\ GMRES~} & \rotatebox{90}{CPU [$\mathrm{s}$]} & \rotatebox{90}{FGMRES~} & \rotatebox{90}{av.\ GMRES~} & \rotatebox{90}{CPU [$\mathrm{s}$]} \\ \hline
$15$ & $2.86\mathrm{e}{5}$& $5$ & $7$ & $114$& $3$ & $3$ & $35$& $3$ & $5$ & $51$\\ \hline 
$31$ & $5.91\mathrm{e}{5}$& $6$ & $6$ & $260$& $4$ & $4$ & $116$& $3$ & $4$ & $98$\\ \hline 
$63$ & $1.20\mathrm{e}{6}$& $8$ & $6$ & $1028$& $4$ & $4$ & $157$& $4$ & $4$ & $399$\\ \hline 
$127$ & $2.42\mathrm{e}{6}$& $9$ & $6$ & $1562$& $4$ & $4$ & $467$& $4$ & $4$ & $502$\\ \hline 
$255$ & $4.86\mathrm{e}{6}$& $9$ & $5$ & $2539$& $5$ & $4$ & $1384$& $4$ & $4$ & $1002$\\ \hline 
$511$ & $9.75\mathrm{e}{6}$& $10$ & $5$ & $*$& $7$ & $4$ & $*$& $5$ & $5$ & $*$\\ \hline 
$1023$ & $1.95\mathrm{e}{7}$& $10$ & $5$ & $*$& $8$ & $4$ & $*$& $6$ & $6$ & $*$\\ \hline 
\end{tabular}
\caption{Solving the Oseen problem with $\widetilde{\mathcal{P}}_C^{(\text{UZ})}$ for $h = h_c$.}
\label{tab:OseenLevel4}
\end{table}

\begin{table}
\centering
\small
\begin{tabular}{|r|r|r|r|r|r|r|r|r|r|r|}
\hline
    \backslashbox{$n_t$}{$\beta$} &  & \multicolumn{3}{c|}{$1\mathrm{e}{-1}$} & \multicolumn{3}{c|}{$1\mathrm{e}{-3}$}  & \multicolumn{3}{c|}{$1\mathrm{e}{-4}$}  \\ \hline
       & \#DOFs & \rotatebox{90}{FGMRES~} & \rotatebox{90}{av.\ GMRES~} & \rotatebox{90}{CPU [$\mathrm{s}$]}  & \rotatebox{90}{FGMRES~} & \rotatebox{90}{av.\ GMRES~} & \rotatebox{90}{CPU [$\mathrm{s}$]} & \rotatebox{90}{FGMRES~} & \rotatebox{90}{av.\ GMRES~} & \rotatebox{90}{CPU [$\mathrm{s}$]} \\ \hline
$15$ & $1.13\mathrm{e}{6}$& $4$ & $6$ & $306$& $3$ & $4$ & $159$& $3$ & $5$ & $186$\\ \hline 
$31$ & $2.33\mathrm{e}{6}$& $6$ & $9$ & $1263$& $4$ & $5$ & $503$& $3$ & $5$ & $281$\\ \hline 
$63$ & $4.73\mathrm{e}{6}$& $6$ & $7$ & $2167$& $4$ & $4$ & $1097$& $3$ & $5$ & $870$\\ \hline 
$127$ & $9.53\mathrm{e}{6}$& $7$ & $7$ & $4724$& $4$ & $5$ & $1999$& $4$ & $4$ & $1489$\\ \hline 
$255$ & $1.91\mathrm{e}{7}$& $8$ & $6$ & $9507$& $5$ & $5$ & $4275$& $4$ & $4$ & $4158$\\ \hline 
$511$ & $3.83\mathrm{e}{7}$& $9$ & $5$ & $*$& $6$ & $4$ & $*$& $4$ & $5$ & $*$\\ \hline 
$1023$ & $7.67\mathrm{e}{7}$& $9$ & $5$ & $*$& $7$ & $4$ & $*$& $5$ & $5$ & $*$\\ \hline 
\end{tabular}
\caption{Solving the Oseen problem with $\widetilde{\mathcal{P}}_C^{(\text{UZ})}$ for $h = h_f$.}
\label{tab:OseenLevel5}
\end{table}

\begin{figure}
\centering
\begin{subfigure}[b]{0.49\textwidth}
\includegraphics[width=\textwidth]{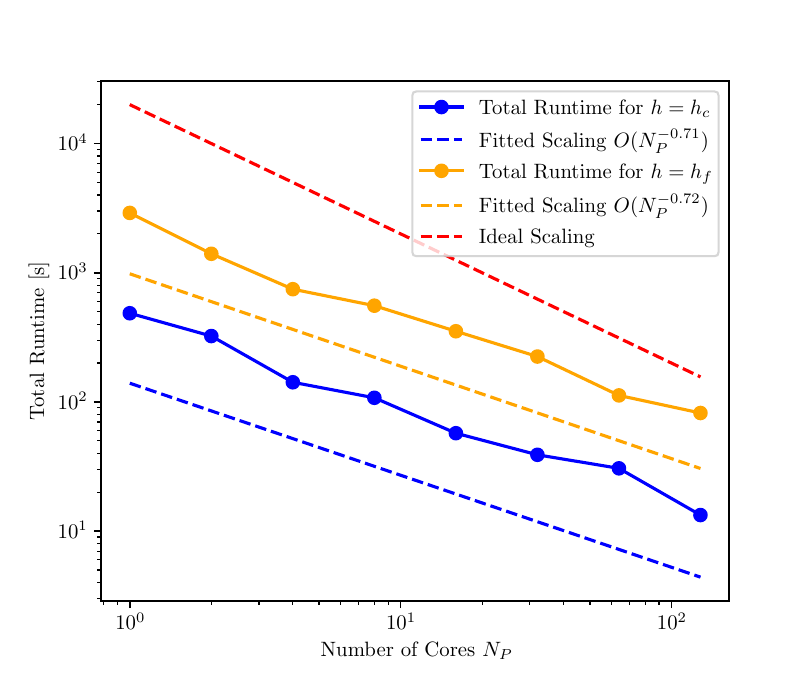}
\caption{Stokes problem}
\label{fig:StrongScaling:Stokes}
\end{subfigure}
\begin{subfigure}[b]{0.49\textwidth}
\includegraphics[width=\textwidth]{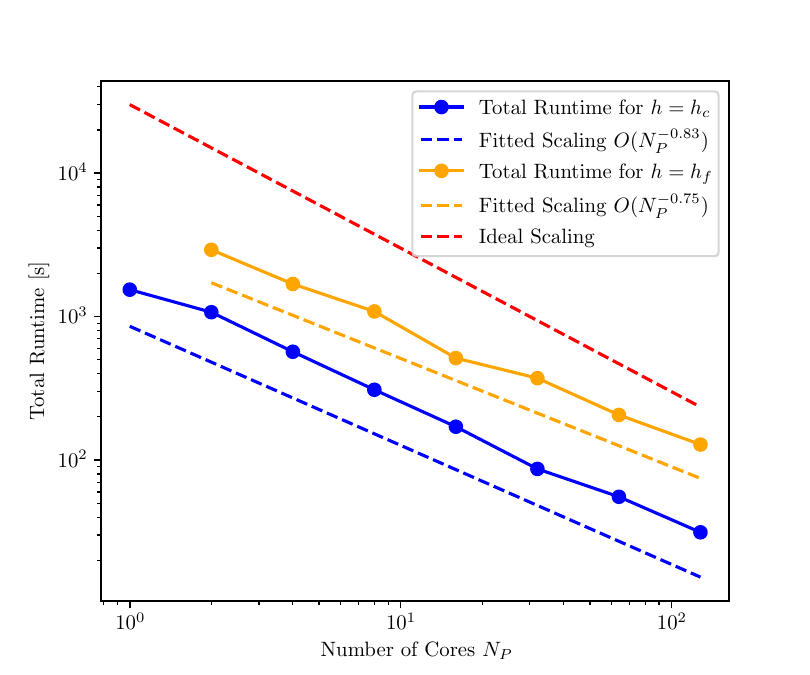}
\caption{Oseen problem}
\label{fig:StrongScaling:Oseen}
\end{subfigure}
\caption{Strong runtime scaling with respect to available number of cores $N_P$.}
\label{fig:StrongScaling}
\end{figure}

\subsection{Parallel implementation}
To test the preconditioners' potential for parallelizability, we present results showing how the runtime of a parallel implementation scales. Based on the preceding results, we will restrict ourselves to the discussed nonlinear preconditioners only. The parallel implementation leverages petsc4py and mpi4py for communication, and the solution of each subblock is carried out in parallel. It is noted that the computation of the FFT is not parallelized. Nevertheless, we are still able to observe consistent scaling since the overall runtime is mainly determined by the solution of individual subblocks rather than the FFT computations.

First, we will consider strong scaling. We set $n_t = 256$, $\beta = 1\mathrm{e}{-3}$, and $T = 5$, and the inner tolerance is set to $\varepsilon = 1\mathrm{e}{-2}$. Figure~\ref{fig:StrongScaling} shows how the runtime scales with the number of cores $N_P$ available to the solver. For the Stokes problem, we achieve scalings of $O(N_P^{-0.71})$ and $O(N_P^{-0.72})$. The runtime scalings for the Oseen problem are slightly better with approximately $O(N_P^{-0.83})$ and $O(N_P^{-0.75})$. This can be explained by the increased complexity of the subblocks, resulting in a longer runtime for solving each of them. The loss in efficiency due to communication and the application of the FFT becomes less significant. A snapshot of the solution to the considered benchmark problem is depicted in Figure~\ref{fig:OseenSampleSolution}. Therein, the timestamp is fixed to $t = T (n_t - 1)/n_t$ and the computations are done for $n_t = 31$.
\begin{figure}
    \centering
    \begin{subfigure}{0.49\textwidth}
\includegraphics[width=\textwidth]{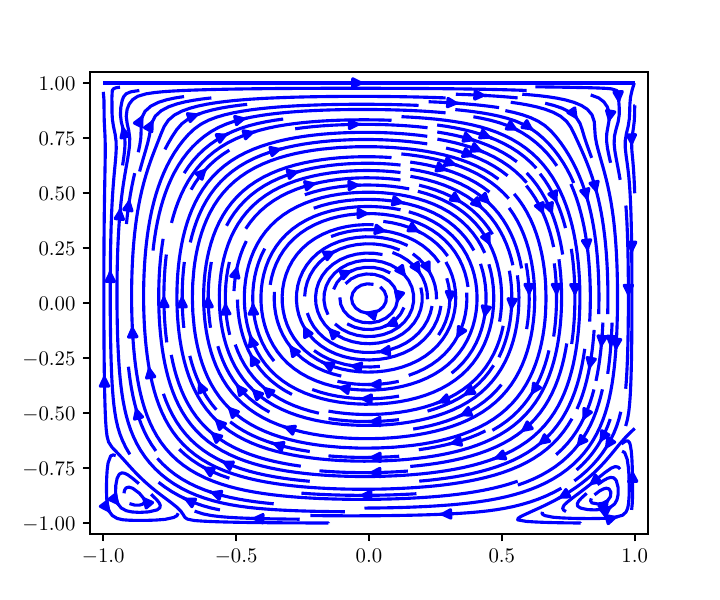}
\caption{$v(x, t)$}
    \end{subfigure}
    \begin{subfigure}{0.49\textwidth}
\includegraphics[width=\textwidth]{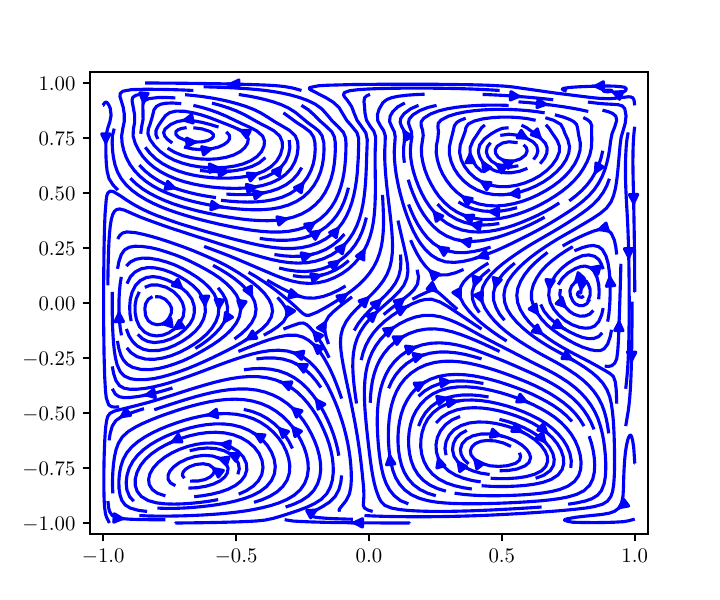}
\caption{$u(x, t)$}
    \end{subfigure}
    \begin{subfigure}{0.49\textwidth}
\includegraphics[width=\textwidth]{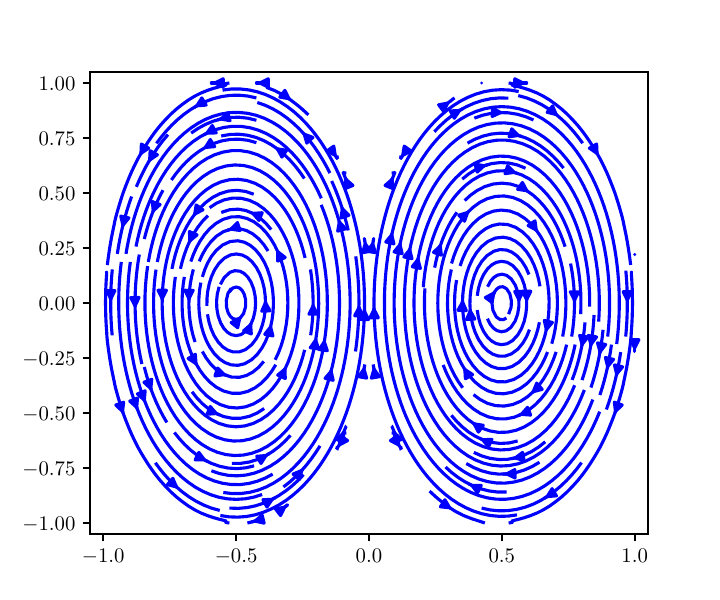}
\caption{$w(x)$}
    \end{subfigure}
    \begin{subfigure}{0.49\textwidth}
\includegraphics[width=\textwidth]{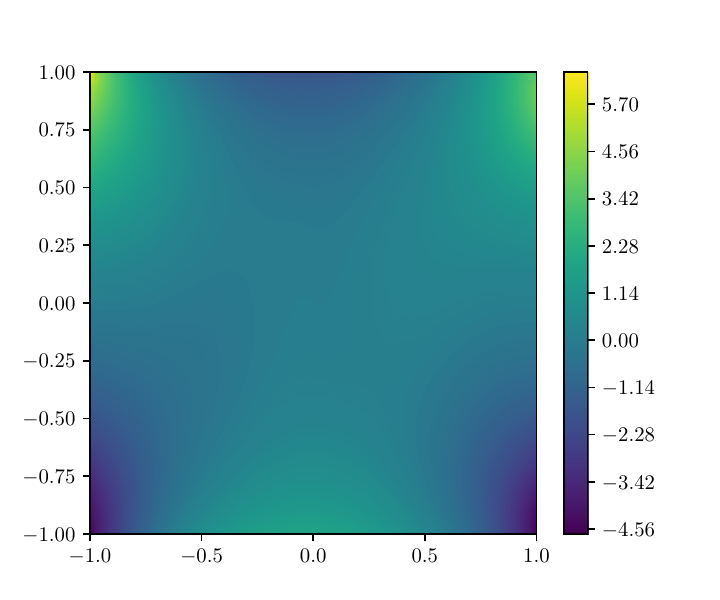}
\caption{$p(x, t)$}
    \end{subfigure}
\caption{Snapshot of the solution to the problem used for benchmarking the strong scaling at timestamp $t = T (n_t - 1)/n_t$.}
\label{fig:OseenSampleSolution}
\end{figure}

Additionally, we consider weak scaling. The standard approach here is to fix all model parameters except for $n_t$ and $N_P$ with the relation $n_t = a N_P$ for some constant $a \in \mathbb{N}$. However, we note it was observed that the first subblock requires significantly longer to solve than the others, so the workload for the first core is reduced by one subblock, i.e., the first core is assigned $a - 1$ subblocks and the others $a$ subblocks. Consequently, for our analysis we set $n_t = a N_P - 1$. The remaining parameters are kept the same as in the previous experiments, with the regularization parameter set to $\beta = 1\mathrm{e}{-3}$. In the case of ideal scaling, we would expect the runtime not to increase significantly with growing $n_t$. Table~\ref{tab:WeakScaling} shows the runtime of the solver for the Stokes and Oseen problem with coarse and fine spatial discretizations. With a factor of $2.87$, the Oseen problem with $h = h_f$ shows the strongest relative increase in runtime from the smallest setting of $n_t$ to the biggest. Hence, a $68.2$ times bigger problem is solved at the cost of a $2.87$ times increase in runtime in the worst-case scenario.

In summary, we have demonstrated that the preconditioners show good strong and weak scaling properties for the Stokes and Oseen problem. Thus, we can solve these large-scale problems efficiently with increasing computing power.

\begin{table}
\centering
\small
    \begin{tabular}{|c||r|r|r|r|r|r|r|} \hline
$n_t$ & $15$ & $31$ & $63$ & $127$ & $255$ & $511$ & $1023$\\ \hline
& \multicolumn{7}{c|}{$h=h_c$} \\ \hline
\# DOFs & $2.86\mathrm{e}{5}$ & $5.91\mathrm{e}{5}$ & $1.20\mathrm{e}{6}$ & $2.42\mathrm{e}{6}$ & $4.86\mathrm{e}{6}$ & $9.75\mathrm{e}{6}$ & $1.95\mathrm{e}{7}$\\ \hline\hline
Stokes & $30$ & $30$ & $32$ & $36$ & $37$ & $56$ & $78$\\ \hline 
Oseen & $46$ & $53$ & $58$ & $64$ & $48$ & $77$ & $72$\\ \hline 
& \multicolumn{7}{c|}{$h=h_f$} \\ \hline
\# DOFs & $1.13\mathrm{e}{6}$ & $2.33\mathrm{e}{6}$ & $4.73\mathrm{e}{6}$ & $9.53\mathrm{e}{6}$ & $1.91\mathrm{e}{7}$ & $3.83\mathrm{e}{7}$ & $7.67\mathrm{e}{7}$\\ \hline \hline
Stokes & $135$ & $141$ & $162$ & $206$ & $295$ & $307$ & $388$\\ \hline 
Oseen & $181$ & $232$ & $255$ & $255$ & $261$ & $303$ & $356$\\ \hline 
    \end{tabular}
\caption{Runtime in seconds of solver for Stokes and Oseen problem. The number of processors is implicitly given by $n_t = 8 N_P - 1$.}
\label{tab:WeakScaling}
\end{table}

\section{Conclusion}
\label{seq:Conclusion}
In this work, we introduced preconditioners for the efficient iterative solution of unsteady Stokes and Oseen control problems. The preconditioners approximate the original problem by its time-periodic equivalent, allowing us to perform a temporal diagonalization and achieve parallel-in-time solvers. The first preconditioner leveraged existing approximations of Stokes problems in order to create a linear preconditioner. To increase efficiency within parallel infrastructures, we introduced additional preconditioners that made use of nested Krylov subspace solvers, and we achieved rapid convergence with these solvers, including for very large problems. Our approach also demonstrated significant robustness with respect to model parameters and the discretization level, as well as very good strong and weak scaling properties.

A pressing challenge for parallel-in-time methods in general is their applicability to nonlinear PDEs. As the application of nonlinear solvers results in time-varying problems, in the case of the problem structures considered in this work this would add the challenge that the diagonalization cannot be carried out in the same way, although mildly nonlinear equations might exhibit fast convergence with modifications of the methods presented. Future work will therefore involve adapting the preconditioners derived in this work to mildly as well as highly nonlinear autonomous flow control problems.

\section{Acknowledgements}
This work has made use of the resources provided by the Edinburgh Compute and Data Facility (ECDF) (\url{http://www.ecdf.ed.ac.uk/}). BH was supported by the MAC-MIGS Centre for Doctoral Training under EPSRC grant EP/S023291/1. JWP was supported by the EPSRC grant EP/S027785/1.

\section{References}
\printbibliography[heading=none]

\end{document}